%% file: June2025_random_edge_Aztec.tex
\documentclass[12pt,a4paper,reqno]{amsart}

\usepackage{euscript,amsfonts,amssymb,amsmath,amscd,amsthm,enumerate,hyperref}
\usepackage{mathtools}

\usepackage{comment}

\usepackage{graphicx}

\usepackage{cleveref}
\usepackage{adjustbox}
\usepackage{color}

\usepackage[margin=1.1in]{geometry}

\usepackage{tikz,bm,float}
\usetikzlibrary{calc}
\colorlet{lgray}{white!85!black}
\colorlet{lred}{white!75!red}

\usepackage{bbm}

\usepackage[margin=1.1in]{geometry}
\newtheorem{theorem}{Theorem} 
\newtheorem*{theorem*}{Theorem}
\newtheorem{lemma}[theorem]{Lemma}

\newtheorem{proposition}[theorem]{Proposition}

\newtheorem{corollary}[theorem]{Corollary}

\theoremstyle{definition}
\newtheorem{definition}[theorem]{Definition}

\theoremstyle{remark}
\newtheorem{remark}[theorem]{Remark}

\numberwithin{equation}{section} \numberwithin{theorem}{section}

\newcommand{\GT}{\mathbb{GT}}

\newcommand{\R}{\mathbb R}

\newcommand{\pa}{\partial}

\newcommand{\mes}{\mathbf m}

\newcommand{\ii}{{\mathbf i}}

\newcommand{\GG}{\mathcal G}

\newcommand{\pr}{\mathrm{pr}}

\renewcommand{\upsilon}{\vartheta}

\newcommand{\rh}{\mathfrak{r}} 

\newcommand{\bb}{\boldsymbol\beta} 
\newcommand{\BB}{\boldsymbol{\mathcal{B}}} 

\usepackage{MnSymbol}

\renewcommand{\Im}{\operatorname{Im}}

\usepackage{skak}

\usepackage{adjustbox}
\usetikzlibrary{arrows}
\usetikzlibrary{decorations.pathreplacing}

\newcommand{\ssp}{\hspace{1pt}}

\newcommand{\of}{:}

\newcommand{\tmmathbf}[1]{\ensuremath{\boldsymbol{#1}}}
\newcommand{\tmop}[1]{\ensuremath{\operatorname{#1}}}

\begin{document}

\title[Domino Tilings in Random Environment]{Domino Tilings of the Aztec Diamond in Random Environment and Schur Generating Functions}

\author{Alexey Bufetov}\address[Alexey Bufetov]{Institute of Mathematics, Leipzig University, Augustusplatz 10, 04109 Leipzig, Germany.}\email{alexey.bufetov@gmail.com}

\author{Leonid Petrov}\address[Leonid Petrov]{Department of Mathematics, University of Virginia,
141 Cabell Drive, Kerchof Hall, Charlottesville, VA, 22902, USA}\email{lenia.petrov@gmail.com}

\author{Panagiotis Zografos}\address[Panagiotis Zografos]{Institute of Mathematics, Leipzig University, Augustusplatz 10, 04109 Leipzig, Germany.}\email{pzografos04@gmail.com}

\begin{abstract}
We study the asymptotic behavior of random domino tilings of the Aztec diamond of size $M$ in a
random environment, where the environment is a one-periodic
sequence of i.i.d.\ random weights attached to domino
positions (i.e., to the edges of the underlying portion of
the square grid).
We consider two cases: either the variance of the
weights decreases at a critical scale $1/M$, or the distribution of the weights is fixed.
In the
former case, the unrescaled
fluctuations of the domino height function are
governed by the
sum of a Gaussian Free Field and an independent Brownian
motion. In the latter case, we establish
fluctuations on the much larger scale $\sqrt M$, given by the Brownian motion alone.

To access asymptotic fluctuations
in random environment, we
employ the method of Schur generating
functions.
Moreover, we substantially
extend the known
Law of Large Numbers and
Central Limit Theorems
for particle systems via Schur generating functions in order
to apply them to our setting.
These results might be of independent interest.
\end{abstract}

\maketitle

\section{Introduction}

\subsection{Overview}

Domino tilings of Aztec diamond were introduced and
enumerated by \cite{elkies1992alternating}. Since then,
random domino tilings of Aztec diamond were found to be a
very rich and exciting research topic. For a uniformly
random choice of a tiling, the existence of the arctic curve
(see \Cref{fig:uniform_200} for an illustration) was first established in
\cite{jockusch1998random}. The local behavior of the tiling
in the neighborhood of the arctic curve was studied in
\cite{Johansson2005arctic} --- it
turned out to be related to the KPZ universality class. The
fluctuations of the height function of the tiling were
established in \cite{chhita2015asymptotic},
\cite{bufetov2016fluctuations} and turn out to be related to the
Gaussian Free Field, in accordance with the Kenyon-Okounkov
conjecture, see \cite{KOS2006},
\cite{OkounkovKenyon2007Limit}.

\begin{figure}[htpb]
\centering
\includegraphics[width=0.7\textwidth]{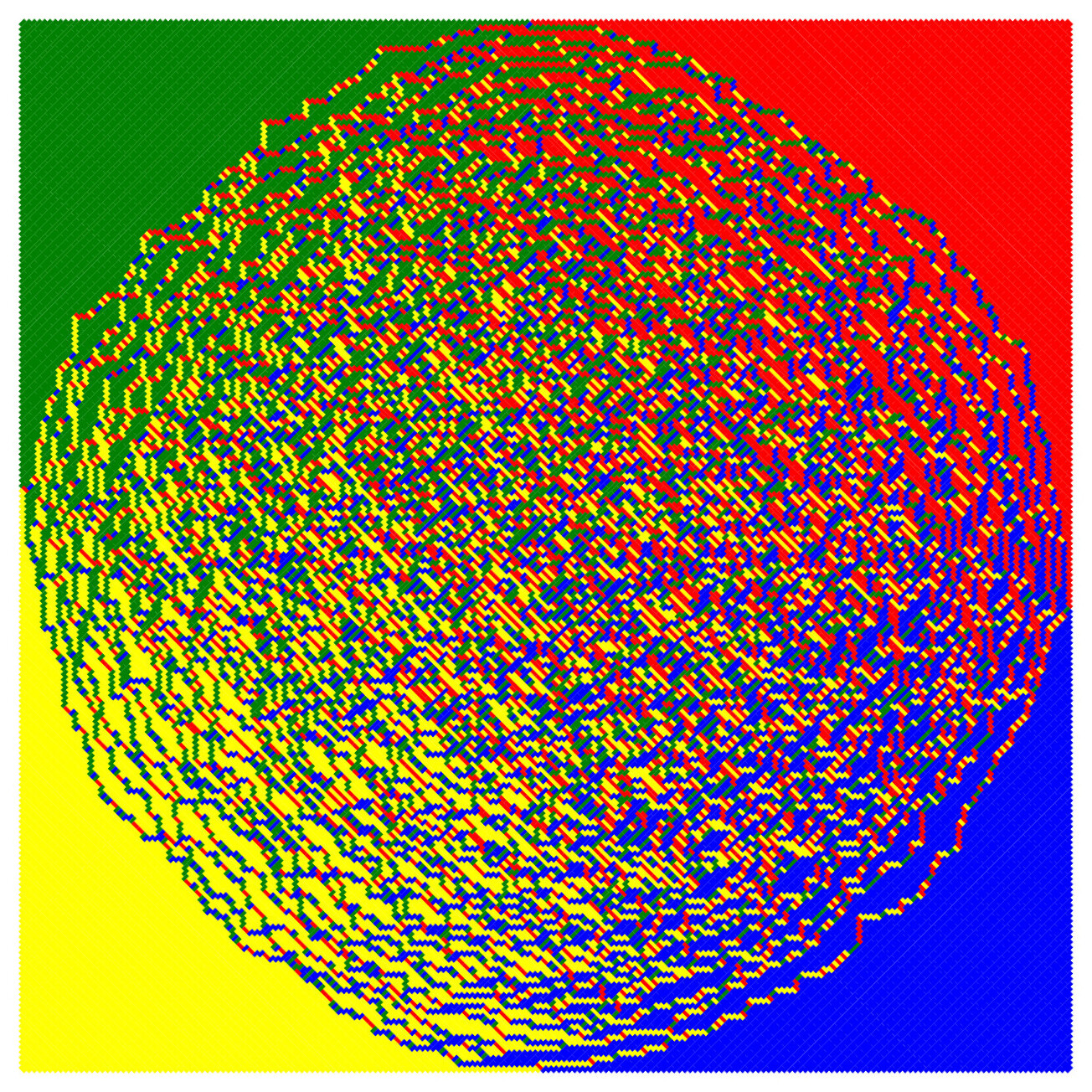}
\caption{A uniformly random domino tiling of the Aztec diamond of size $M=200$.}
\label{fig:uniform_200}
\end{figure}

Much attention has also been devoted to non-uniform random
tilings. A standard choice, inspired by statistical
mechanics, is to assign a positive weight to each potential
domino position in the Aztec diamond; and to select a tiling
with probability proportional to the product of the weights
of all dominos it contains.

A completely arbitrary choice of such weights seems to lead
to random domino tilings which currently available tools are
not strong enough to analyze in detail. However, several
specific choices of weights were thoroughly analyzed in the
last twenty years, leading, in particular, to beautiful
connections with other fields of mathematics. The simplest
out of these choices is a one-periodic weighting. It is well understood that
such a choice of weights leads to a Schur measure on Young
diagrams, the object that was introduced in and studied in
\cite{okounkov2001infinite}, \cite{okounkov2003correlation}.

\begin{figure}
\includegraphics[width=\textwidth]{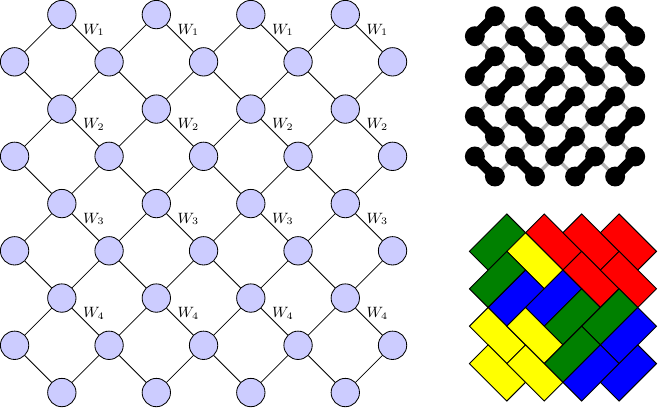}
\caption{Left: One-periodic weights $W_i$ attached to the edges of the Aztec diamond graph
of size $M=4$. The edges which are not labeled have weight $1$.
We assume that $W_i$'s form an i.i.d.\ random sequence with common distribution $\BB$.
Right: An example of a
dimer configuration on the Aztec diamond graph, and the corresponding domino tiling.
Given the weights $W_i$, this configuration has conditional probability proportional to $W_1^3W_2^2$.
In terms of domino colors, we assign nontrivial Boltzmann weights to red dominos.}
\label{fig:AztecWeights-intro}
\end{figure}

Even richer phenomena emerge in random domino tilings of the Aztec diamond when two-periodic or, more generally, multi-periodic weights are considered; see, for instance,
\cite{chhita2015asymptotic},
\cite{ChhitaYoung2014},
\cite{chhita2016domino},
\cite{duits2017two},
\cite{Berggren2021},
\cite{BerggrenBorodin2023},
\cite{KenyonPrause2024},
\cite{BergrenBorodin2024crystallization},
\cite{boutillier2024fock},
\cite{bobenko2024dimers},
\cite{mason2022two}
for an (incomplete) list of recent results in this direction.

In the present work we move in a different direction by
allowing the weights themselves to be random. We focus on
what is arguably the simplest nontrivial setting: a
one-periodic model in which the parameters $W_1,W_2,\dots,W_M$
are i.i.d.\footnote{Independent identically distributed.}
with a common distribution $\BB$ (see \Cref{sub:Aztec-intro} below for a precise definition,
and
\Cref{fig:AztecWeights-intro}, left, for an illustration of the weighting).
As far as we are aware, there
is essentially no mathematical literature on dimer models
with random edge weights. Related studies in mathematical
physics, such as \cite{perret2012super}, treat the Aztec
diamond with all weights chosen i.i.d., but the specific
model examined here seems to be new.
As we were preparing this manuscript, we learned
of an ongoing work
\cite{DuitsVanPeski2025}
on a different
model involving domino tilings of the Aztec diamond with
random edge weights.

\medskip

We distinguish two regimes for domino tilings of the Aztec diamond of size $M\to\infty$ with random weights:
\begin{enumerate}[$\bullet$]
\item \emph{Critical vanishing variance.}  The variance of
$\BB=\BB_M$ decays like $M^{-1}$, so that the fluctuations
produced by the randomness of the weights and those present
in the uniform model occur on the same scale. These critically scaled random weights
do not change the limit shape of the height function
of the domino tiling (\Cref{thm:limit_shape_decreasing_variance}).
We establish the Central Limit Theorem (\Cref{th:main-decreasing,th:main-decreasing-multilevel}).
See \Cref{th:main-decreasing-intro} for an informal version of the statement.
\item \emph{Fixed variance.}  The distribution $\BB$ is independent of $M$.  Here, we prove both a Law of Large Numbers and a Central Limit Theorem for the height function; see \Cref{LLN:tilings-fixed} and \Cref{th:CLT-tilings},
as well as
\Cref{th:main-fixed-intro} for an informal version of the statements.
\end{enumerate}
Both families of results are universal in the sense that they
do not depend on the specific distribution $\BB$ of the random weights
(provided that the distribution satisfies suitable mild conditions).

To establish these asymptotic results, we employ the method of Schur
generating functions developed in
\cite{GorinBufetov2013free}, \cite{bufetov2016fluctuations},
\cite{BufetovGorin2019};
One of important applications of
the method of Schur generating functions
is the
analysis of global fluctuations of uniformly random domino tilings of the Aztec
diamond (performed in \cite{bufetov2016fluctuations}).
Further applications and extensions of the method of Schur generating functions
include, for example,
\cite{BufetovKnizel2018} (domino tilings of more general domains),
\cite{borodin2019product}, \cite{gorin2022gaussian} (random matrix models with unitary symmetry),
and \cite{huang2021law}, \cite{cuenca2025discrete} (extension to Jack generating functions, with applications to
general Beta random matrix models).

The Schur generating function techniques of
\cite{GorinBufetov2013free},
\cite{bufetov2016fluctuations}
also cover deterministic
one-periodic weights. 
Moreover, these techniques also cover the case of decreasing
variance of random weights, as we show in
\Cref{sec:decreasing-variance}. 
The reason for this is that the fluctuations are of the same
order as in the case of deterministic weights, even though
the distribution of the fluctuations
changes.
On the other hand, the presence of random weights with a
fixed distribution $\BB$ (the second case above) involves markedly different
scalings and lies beyond the scope of the known works, so we
must significantly extend their general results. The Law of
Large Numbers (\Cref{th:LLN}) and Central Limit
Theorem (\Cref{th:CLT-gen}) for Schur generating
functions obtained here constitute the principal
contributions of the present paper. Although we apply
them only for the Aztec diamond, we expect these theorems to
be useful in other settings.

Let us remark that, although the determinantal process approach (either via asymptotics of the kernel, or through connections to orthogonal polynomials) is one of the standard methods for studying dimer models, a dimer model with generic random weights produces a point process that is likely not determinantal.

\subsection{Domino tilings with random edge weights. Model and results}
\label{sub:Aztec-intro}

The Aztec diamond graph is a classical bipartite graph
embedded in the square lattice; see the left panel of
\Cref{fig:AztecWeights-intro}. Let $M$ denote its size
(order), that is, the number of lattice vertices along each
side.
Dimer coverings (which are the same as perfect matchings) of this graph
are essentially the same as tilings of the Aztec diamond
by $1\times 2$ (or $2\times 1$) dominos.
We refer to \Cref{fig:AztecWeights-intro} for the
correspondence.
We equip the edges of the Aztec diamond graph with
one-periodic weights $(W_1,\dots,W_M)$ shown in
\Cref{fig:AztecWeights-intro}. Here, ``one-periodic'' means
that the weights change in one direction, but are repeated
in the other.
The weights
$W_1,\dots,W_M$ are assumed to be independent and
identically distributed according to a probability law $\BB_M$ on $\mathbb{R}_{>0}$ (possibly
depending on $M$).
Our aim is to study the asymptotic
behavior, as $M\to\infty$, of a random domino tiling on the
graph with these weights. In detail, we consider a probability measure
on the set of all dimer coverings (perfect matchings) $D$ of the Aztec diamond graph
defined as
\begin{equation*}
	\operatorname{\mathbf{P}}( D )\coloneqq\frac{1}{Z}\prod_{e\in D} \operatorname{weight} (e),
\end{equation*}
where $\operatorname{weight}(e)$ is either equal to $1$, or is one of the random weights $W_i$.

\begin{figure}[htpb]
\centering
\includegraphics[height=0.42\textwidth]{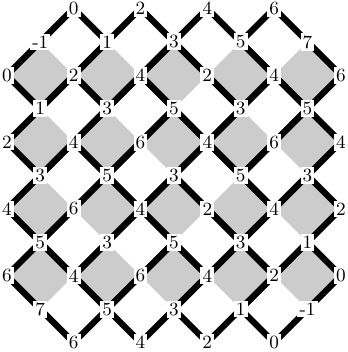}
\qquad
\includegraphics[height=0.42\textwidth]{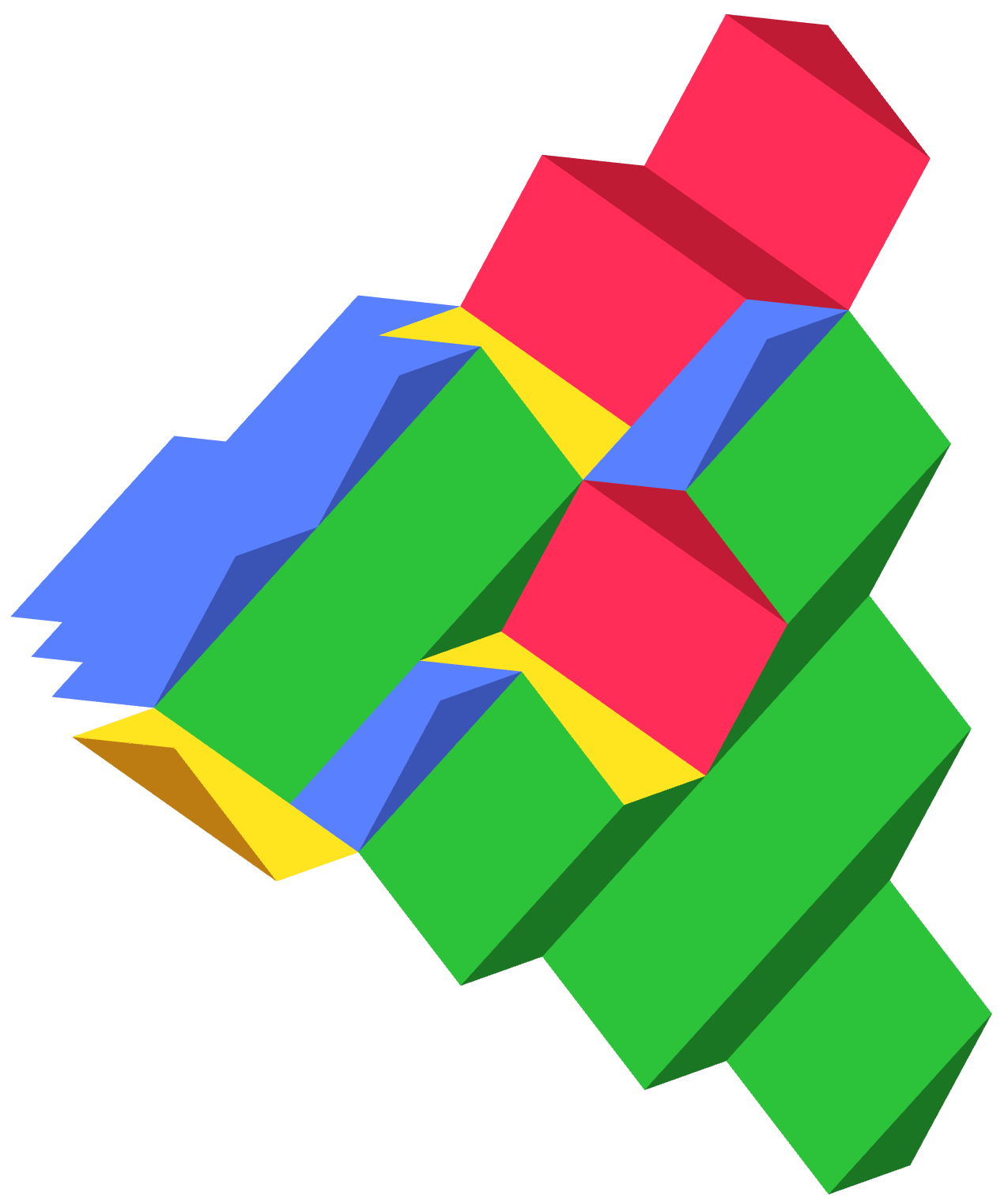}
\caption{A domino tiling of the Aztec diamond of size $M=4$ together with the corresponding height function
and its 3D plot. Interactive 3D plots of domino height functions is available online at
\cite{petrov2025Aztec_general}.
The height function changes by $\pm 1$ along the domino edges, and by $\pm 3$ across the
dominos. The sign of the change depends on the parity with respect to the checkerboard coloring of the
cells of $\mathbb{Z}^2$. The height function (up to a global shift) determines the domino tiling uniquely.}
\label{fig:height-function}
\end{figure}

Each domino tiling of the Aztec diamond determines a height function \cite{thurston1990conway}.
See \Cref{fig:height-function} for an illustration.
We describe the asymptotic behavior of our random domino tilings through this height function:

\begin{theorem}[Informal version of \Cref{thm:limit_shape_decreasing_variance,th:main-decreasing,th:main-decreasing-multilevel}]
\label{th:main-decreasing-intro}
Assume that the variance of the distribution $\BB_M$ decays
at rate $1/M$, and that the corresponding random weights $W_i=W_i(M)$ converge in probability
to a deterministic constant $w>0$. Then the height function
has the same limit shape as
in the case of fixed deterministic weights $W_i\equiv w$.\footnote{In particular, the
arctic curve in this case is an ellipse \cite{jockusch1998random} inscribed into the square bounding the Aztec diamond.}
Furthermore, the
fluctuations of the height function in
suitable coordinates are described by the sum of two
independent random objects:
(a deterministic image of)
the
Gaussian Free Field, and a
one-dimensional Brownian motion.
\end{theorem}

Kenyon-Okounkov conjecture \cite{Kenyon2001GFF},
\cite{Kenyon2004Height}, \cite{OkounkovKenyon2007Limit}
generally predicts Gaussian Free Field fluctuations of
random tilings in a variety of models, see \cite[Ch.
11]{gorin2021lectures} for an exposition. This behavior was
established in several tiling (dimer) models, see, e.g.,
\cite{Kenyon2001GFF},
\cite{Kenyon2004Height},
\cite{BorFerr2008DF},
\cite{Petrov2012GFF},
\cite{Berest-Laslier-2016},
\cite{BerggrenNicoletti2025}.
For domino tilings of the Aztec diamond with nonrandom unit edge weights
the Gaussian Free Field fluctuations were established in  \cite{chhita2015asymptotic}, \cite{bufetov2016fluctuations}.

We show that introducing random edge weights alters the
fluctuations of the height function even when the variance
of the weight distribution tends to zero. Under the scaling
of the variance by $1/M$, as in
\Cref{th:main-decreasing-intro},
the new fluctuations produced by random weights and the
``old'' fluctuations arising from fixed weights (which give
the Gaussian Free Field) occur on the same scale, and
contribute independently to the total fluctuations of the height function.

\begin{theorem}[Informal version of \Cref{LLN:tilings-fixed} and \Cref{th:CLT-tilings}]
\label{th:main-fixed-intro}
Assume that the distribution $\BB_M$ is independent of $M$
and possesses moments of all orders. Then the height
function satisfies a Law of Large Numbers. Moreover, its
fluctuations occur on scale $\sqrt{M}$ and, in suitable
coordinates, are described by a one-dimensional Brownian
motion.
\end{theorem}

Note that in this setting (typical for processes in random
environment, when the edge weights are sampled once from a
fixed distribution independent of $M$), the height function
behaves very differently from the fixed-weight case.
In general, its limit shape does not coincide with that of
any fixed-weight domino tiling of the Aztec diamond (see
\Cref{fig:Bernoulli_and_fixed,fig:uniform_random_0_2_300}
for illustrations; the first figure shows a special case in
which Bernoulli random edge weights yield the same limit
shape as a tiling with deterministic periodic weights).
Moreover, the fluctuations of the height function are now
much larger --- of order $\sqrt{M}$, in contrast to the
unnormalized fluctuations in \Cref{th:main-decreasing-intro}
--- and are governed by the Brownian motion.
Intuitively this may be explained by the fact that
by the classical Central Limit Theorem, linear statistics of $M$ i.i.d.\ random variables fluctuate on the scale $\sqrt{M}$.
These classical Gaussian fluctuations 
dominate the much smaller fluctuations
produced by the randomness of the domino tiling with fixed weights.

\begin{figure}[htpb]
\centering
\includegraphics[width=0.48\textwidth]{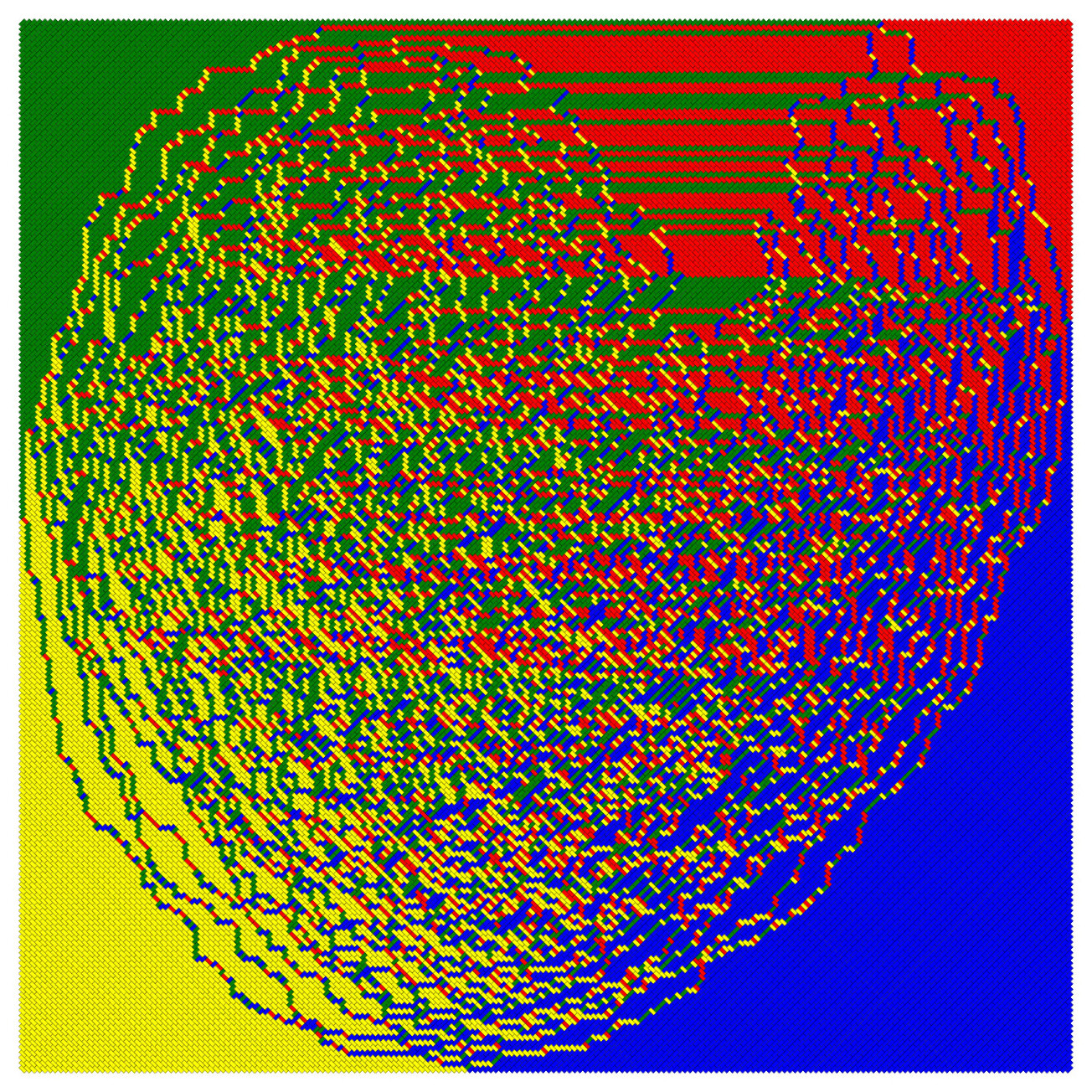}
\quad
\includegraphics[width=0.48\textwidth]{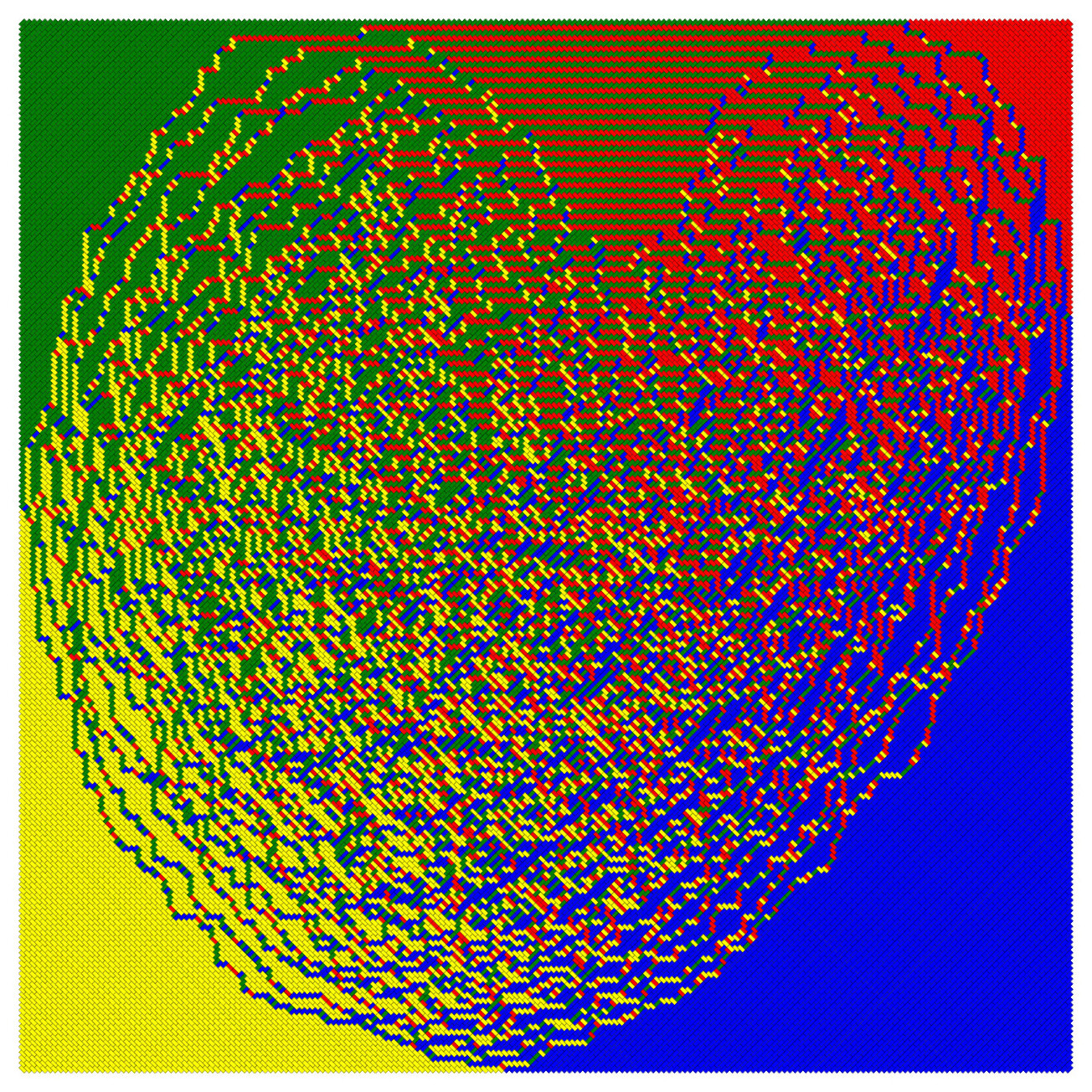}
\caption{Left: A random domino tiling of the Aztec diamond of size $M=200$ with i.i.d.\ Bernoulli weights with
$\operatorname{\mathbf{P}}(W_i=\frac12)=\operatorname{\mathbf{P}}(W_i=5)=\frac{1}{2}$. Right:
A random domino tiling of the Aztec diamond of size $M=200$ with fixed nonrandom weights $W_i=\frac{1}{2}$ and $W_i=5$ for odd and even $i$,
respectively. The limit shapes for these two choices of weights --- random and deterministic periodic with the
same values of $W_i$ in the same proportions --- are the same.
See also \Cref{fig:limit-shape-examples} for limit shapes and arctic curves which we derive in this work.
These samples were generated with the domino shuffling algorithm of \cite{propp2003generalized};
the implementation we used is available online at \cite{petrov2025random_edge_Aztec_simulation}.}
\label{fig:Bernoulli_and_fixed}
\end{figure}

\begin{figure}[htpb]
\centering
\includegraphics[width=0.48\textwidth]{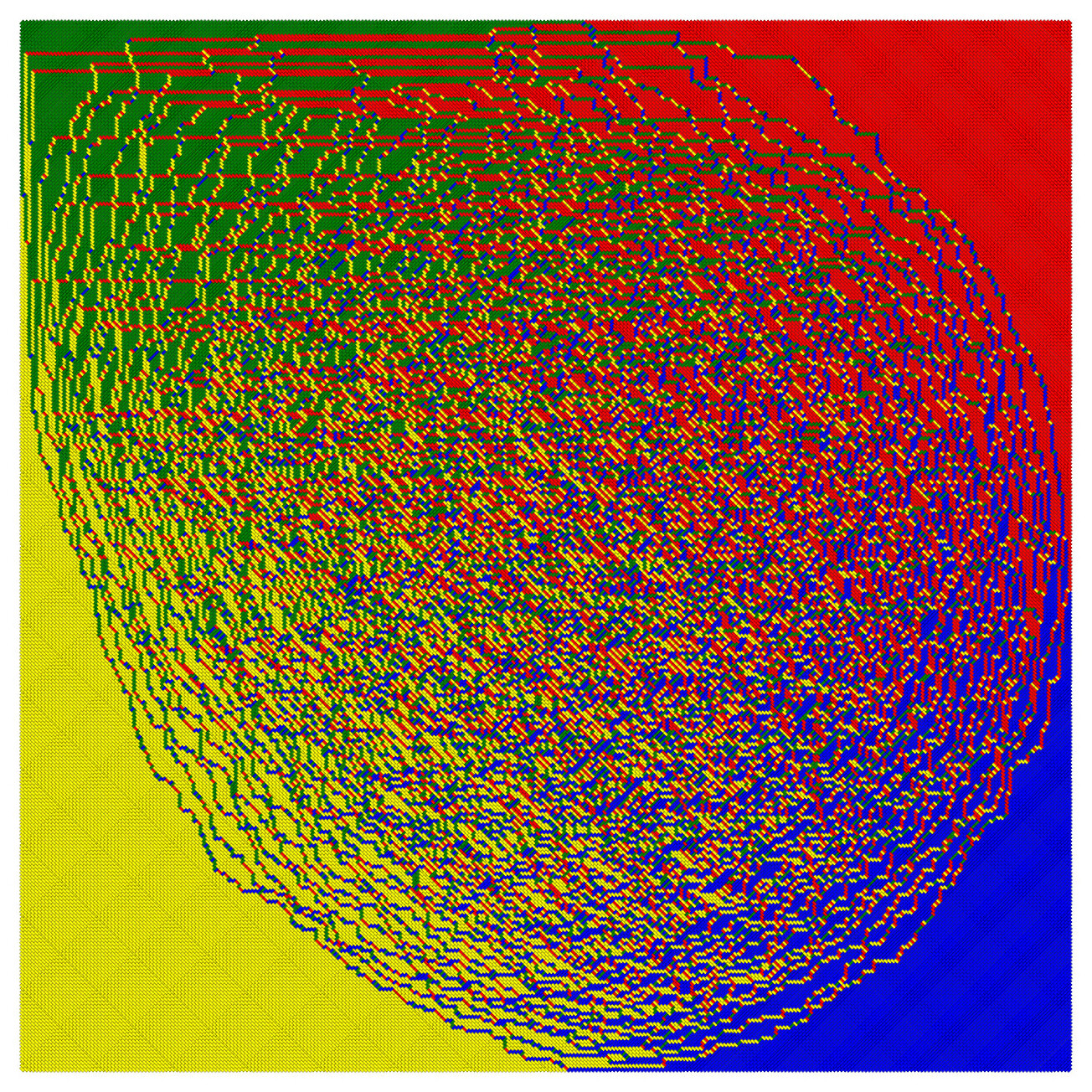}
\quad
\includegraphics[width=0.48\textwidth]{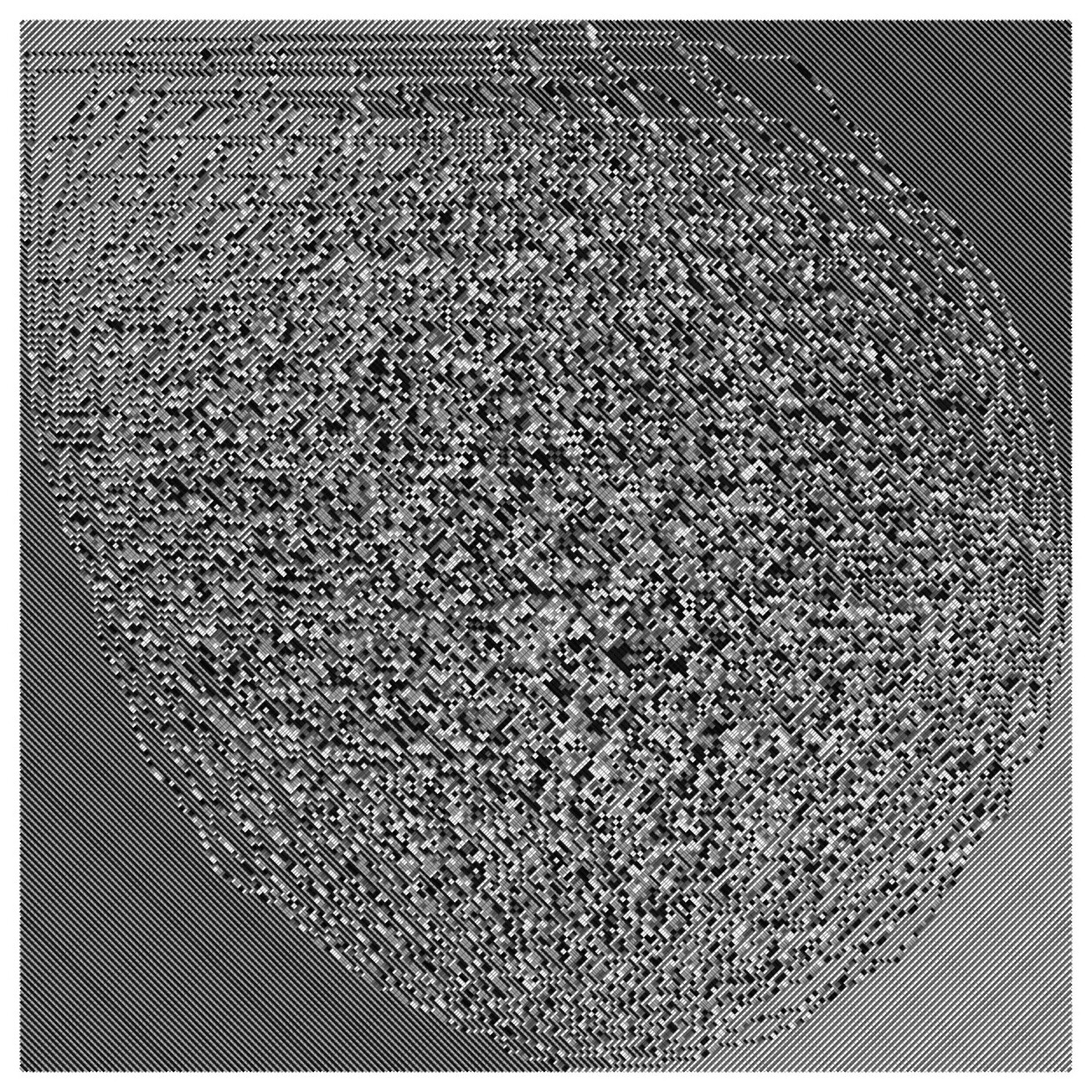}
\caption{A random domino tiling of the Aztec diamond of size $M=300$ with i.i.d.\ continuous uniform random weights $W_i$
on $[0,2]$. The right image is a grayscale version with eight shades depending on the parity of the dominos with respect to the lattice.
See also \Cref{fig:limit-shape-examples} for limit shapes and arctic curves which we derive in this work.
These samples were generated with the domino shuffling algorithm of \cite{propp2003generalized};
the implementation we used is available online at \cite{petrov2025random_edge_Aztec_simulation}.}
\label{fig:uniform_random_0_2_300}
\end{figure}

\subsection{Asymptotics via Schur generating functions. General results}
\label{subsec:Schur-generating-functions-intro}

Our key technical tool for the asymptotic analysis
of domino tilings of the Aztec diamond with random edge weights
is the method of Schur generating functions.
We refer to \Cref{prop:SGF-rand} for a precise statement connecting
the domino tiling model to Schur generating functions
(and also to
\Cref{rem:SGF-rand} for comparison with other approaches).

In this subsection, we present the most technical part of our results,
which concern asymptotic behavior
of particle systems
via Schur generating functions.
They extend the results of
\cite{GorinBufetov2013free},
\cite{bufetov2016fluctuations},
and
\cite{BufetovGorin2019},
which were not sufficient for our model with random edge weights.

An $N$-tuple of non-increasing integers $\lambda = (\lambda_1\ge \lambda_2 \ge \dots \ge \lambda_N)$
is called a \textit{signature} of length $N$. We denote by $\GT_N$ the set of all
signatures of length $N$.
Let
$s_{\lambda}(x_1, \dots, x_N)$
be the \emph{Schur function} parameterized by $\lambda$
More precisely, $s_\lambda(x_1,\ldots,x_N)$
is a homogeneous symmetric Laurent polynomial in the $x_i$'s
of degree $|\lambda|=\lambda_1+\ldots+\lambda_N $
(see \eqref{eq:Schur-function-def} for the full definition).
Let $\rho_N$ be a probability measure on the set $\GT_N$.
A \emph{Schur generating function}
$S_{\rho_N} (x_1,\dots,x_N)$ is a symmetric Laurent power series in $x_1,\dots,x_N$ given by
$$
 S_{\rho_N} (x_1,\dots,x_N) \coloneqq \sum_{\lambda \in \GT_N} \rho_N (\lambda)
 \ssp\frac{s_\lambda(x_1,\dots,x_N)}{s_\lambda(1^N)}.
$$
Define the empirical measure
\begin{equation*}
  m [ \rho_N ] \coloneqq \frac{1}{N} \sum_{i = 1}^N \delta \left(
  \frac{\lambda_i + N - i}{N} \right), \qquad \mbox{where $\lambda=(\lambda_1, \dots, \lambda_N)$ is $\rho_N$-distributed.}
\end{equation*}
If $F_{m[\rho_N]}(x)$ denotes the cumulative distribution function of $m [\rho_N]$,
then its rescaling of the form $N\ssp F_{m[\rho_N]}(x/N)$ corresponds to the (integer-valued) height function
of the domino tiling of the Aztec diamond (illustrated in \Cref{fig:height-function}) along a given one-dimensional slice.
We refer to \Cref{subsec:comb} for details on the correspondence between
particle configurations and domino tilings.

Asymptotic behavior of the empirical measure $m [\rho_N]$ can be
understood by looking at the
asymptotics of its Schur generating function,
as $N \to \infty$.
Our first general result is as follows:
\begin{theorem}[Law of Large Numbers]
  \label{th:LLN-intro} Let $\rho_N$, $N \in \mathbb{Z}_{\ge1}$, be a sequence of probability
  measures on the sets $\GT_N$. Assume that there exists a sequence
  of symmetric functions $F_k (z_1, \ldots, z_k)$, $k \ge1$, analytic
  in a (complex) neighborhood of $1^k\coloneqq (1,\ldots,1 )\in
	\mathbb{C}^k$,\footnote{Throughout the paper, we use the notation $1^m$
	to denote the vector $(1,\ldots,1)$ of $m$ ones, where $m\ge0$ is an arbitrary integer,
	and similarly for $0^m$.}
	such that
  \begin{equation}
		\label{assumptioN1-intro}
    \lim_{N \rightarrow \infty}  \sqrt[N]{S_{\rho_N} (u_1,
    \ldots, u_k, 1^{N - k})} = F_k (u_1, \ldots, u_k),
  \end{equation}
  where
	$k$ is fixed, and the
	the convergence is uniform in a complex neighborhood of $1^k$. Then,
	as $N\to\infty$,
	the
  random measure $m [\rho_N]$ converges
	in probability, in the sense of moments,\footnote{That is, all moments of the empirical measure $m [\rho_N]$,
	which are random variables, converge in probability to the moments of the limiting measure.}
	to a deterministic probability measure
  $\tmmathbf{\mu}$ on $\mathbb{R}$.
	Moreover, the moments $(\tmmathbf{\mu}_k)_{k\ge1}$ of the limiting measure $\tmmathbf{\mu}$ are
	given by
  \begin{multline}
  \tmmathbf{\mu}_k = \sum_{l = 0}^k \binom{k}{l} \frac{1}{(l + 1) !}
     \\ \times
     \pa_u^{\ssp l} \left[ (1+u)^k \left( \pa_1 F_{l+1} (1+u, 1+u w_{l+1}, 1+u w_{l+1}^2, \ldots, 1+u w_{l+1}^l ) \right)^{k-l} \right] \Big\vert_{u=0}, \label{eq:LLN-gen-intro}
  \end{multline}
 where $w_{m} \coloneqq \exp \left( \frac{2 \pi \sqrt{-1}}{m}
 \right)$ is the $m$-th root of unity, $\pa_u$ is the partial derivative with respect to $u$,
 and $\pa_1$ is the partial derivative with respect to the first variable in the function $F_{l+1}$.
\end{theorem}
\Cref{th:LLN-intro} provides a substantial generalization of Theorem 5.1 in \cite{GorinBufetov2013free}. We recover it in Corollary \ref{cor:sanity} below.

\medskip
Our second general result is on fluctuations of
the empirical measure $m [\rho_N]$ determined
by the Schur generating function.
For a probability measure $\rho_N$ on $\GT_N$, define
the $k$-th moment as
\begin{equation}
\label{eq:p_k-rho-intro}
	p_k^{(\rho_N)} \coloneqq \sum_{i=1}^N \left( \lambda_i +N-i
	\right)^k, \qquad \mbox{where $\lambda=(\lambda_1, \dots,
	\lambda_N)$ is $\rho_N$-distributed.}
\end{equation}
\begin{theorem}[Central Limit Theorem]
  \label{th:Gaussian-intro}
	Under the assumptions of
	\Cref{th:LLN-intro},
	the vector
	\begin{equation}
	\label{eq:scale-CLT-intro}
		\left( \frac{ p_k^{(\rho_N)} -
		\raisebox{-1pt}{$\operatorname{\mathbf{E}}$} [
		p_k^{(\rho_N)} ] \mathstrut}{N^{k+1/2 } } \right)_{k
		\ge 1}
	\end{equation}
	converges to a mean-zero Gaussian vector with covariance given by formula \eqref{eq:cov-general-statement} below.
\end{theorem}

\begin{remark}
\label{rmk:Gaussian-intro-normalization-discussion}
The principal distinction between our setting and those in
\cite{bufetov2016fluctuations}, \cite{BufetovGorin2019}
lies in the fluctuation scale.
Namely, these earlier works considered unnormalized fluctuations of the (integer-valued)
height function $N\left( F_{m[\rho_N]}(y) - \raisebox{-1pt}{$\operatorname{\mathbf{E}}$} [F_{m[\rho_N]}(y)] \right)$,
where $y$ belongs to a compact interval. Then we have
\begin{equation}
\label{eq:unscaled-moments-intro}
\frac{p_k^{(\rho_N)}- \raisebox{-1pt}{$\operatorname{\mathbf{E}}$} [p_k^{(\rho_N)}]}{N^k}=
\int_{-\infty}^{\infty} y^k \ssp d\left[
N\left( F_{m[\rho_N]}(y) - \raisebox{-1pt}{$\operatorname{\mathbf{E}}$} [F_{m[\rho_N]}(y)] \right)
\right],
\end{equation}
and in \cite{bufetov2016fluctuations}, these moments are shown to
have jointly Gaussian asymptotics
(we recall one of these results in \Cref{th:BG2-one-level}
below).

In contrast, for random domino tilings of the Aztec diamond with random edge weights
with a fixed distribution $\BB$, the denominators in \eqref{eq:unscaled-moments-intro} must be replaced by
$N^{k+1/2}$. This means that the fluctuations of the height function are now growing on the scale $\sqrt{N}$.
This new
scaling alters the behavior of the Schur generating
functions, a feature captured by \Cref{th:Gaussian-intro}.
\end{remark}

Our proofs of \Cref{th:LLN-intro,th:Gaussian-intro} largely
follow the approach of \cite{GorinBufetov2013free},
\cite{bufetov2016fluctuations}, \cite{BufetovGorin2019};
for this reason we keep the proofs of our theorems
relatively brief,
focusing on the modifications necessitated by the new setting.
The
moment computation for \Cref{th:LLN-intro} required several new
ideas, most notably the use of the complex domain. The argument
for \Cref{th:Gaussian-intro} is closer to that in
\cite{bufetov2016fluctuations} and, in some respects, even
simpler. Nevertheless, the new fluctuation scale and the
significantly broader assumptions introduced additional
challenges that we had to overcome.

\subsection*{Acknowledgments}
We are grateful to Alexei Borodin, Vadim Gorin, Kurt Johansson, and Sasha Sodin for helpful comments.
We used the domino shuffling code by Sunil Chhita (ported to JavaScript \cite{petrov2025random_edge_Aztec_simulation}
by the second author) in order to produce the domino tiling samples of large Aztec diamonds.
A.~Bufetov and P.~Zografos were partially supported by the European Research Council (ERC), Grant Agreement No. 101041499.
L.~Petrov was partially supported by the NSF grant DMS-2153869 and by the Simons Collaboration
Grant for Mathematicians 709055.

\section{Preliminaries}

\subsection{Signatures and Schur functions}

An $N$-tuple of non-increasing integers $\lambda = (\lambda_1\ge \lambda_2 \ge \dots \ge \lambda_N)$
is called a \textit{signature} of length $N$.
We denote by $\GT_N$ the set of all
signatures of length $N$.
Let also $|\lambda|\coloneqq\sum_{i=1}^N \lambda_i$.
Signatures $\lambda \in \GT_N$ and $\mu \in \GT_{N-1}$
\textit{interlace} (notation $\mu \prec \lambda$), if
$\lambda_i \ge \mu_i \ge \lambda_{i+1}$, for all $i=1,
\dots, N-1$. Signatures $\nu \in \GT_N$ and $\lambda \in
\GT_{N}$ \textit{interlace vertically} (notation $\lambda
\prec_v \nu$), if $\nu_i - \lambda_i \in \{0,1 \}$ for all
$i=1, \dots, N$.

The \textit{Schur function} $s_{\lambda}$, $\lambda\in\GT_N$, is a
symmetric Laurent polynomial
in $x_1,\ldots,x_N $ of degree $|\lambda|$
defined by
\begin{equation}
\label{eq:Schur-function-def}
s_\lambda (x_1, \dots ,x_N)\coloneqq
\frac{\det \left[x_i^{ \lambda_j + N -j}\right] }{ \det \left[x_i^{N-j}\right]}
=
\frac{\det\left[x_i^{\lambda_j+N-j}\right]}{\prod \limits_{i<j}
 (x_i-x_j)}.
\end{equation}
Schur functions $s_\lambda$ form a linear basis in the space of symmetric Laurent polynomials
in $x_1,\ldots,x_N$, where $\lambda$ runs over $\GT_N$.

\subsection{Domino tilings and sequences of signatures}
\label{subsec:comb}

For each $t=1,2,\dots, M$, let $\lambda^{(t)}$,
$\upsilon^{(t)}$ be signatures of length $t$.
\begin{definition}
\label{def:beta_vertical_transitions}
	Let the coefficients $\kappa_{\beta} (\lambda^{(t)} \to
	\upsilon^{(t)})$ be defined via the following expansion
	into the linear basis of Schur functions:
	\begin{equation}
	\label{eq:Schur-branching-kappa-coef}
		\frac{s_{\lambda^{(t)}} (x_1, \dots, x_t)}{s_{\lambda^{(t)}} (1^t)}
		\ssp \prod_{i=1}^t
		(1-\beta + \beta x_i)
		=
		\sum_{\upsilon^{(t)} \in \GT_t} \kappa_{\beta} (\lambda^{(t)} \to
		\upsilon^{(t)}) \ssp \frac{s_{\upsilon^{(t)}} (x_1, \dots, x_t)}{s_{\upsilon^{(t)}}
		(1^t)}.
	\end{equation}
	It follows from the Pieri rule (see, e.g., \cite[Lemma 2.12]{BufetovKnizel2018}) that
	these coefficients have the explicit form
	\begin{equation} \label{st}
	 \kappa_{\beta} (\lambda^{(t)} \to \upsilon^{(t)}) = \begin{cases} \beta^{|\upsilon^{(t)}| - |\lambda^{(t)}| } (1-\beta)^{t-(|\upsilon^{(t)}| - |\lambda^{(t)}| ) } \ssp
	 \frac{s_{\upsilon^{(t)}}(1^t)}{s_{\lambda^{(t)}}(1^t)}, & \lambda^{(t)} \prec_v
		 \upsilon^{(t)} ; \\
	 0, & {otherwise},\end{cases}
	\end{equation}
	By setting $x_i\equiv 1$ in \eqref{eq:Schur-branching-kappa-coef},
	we see that $\kappa_{\beta} (\lambda^{(t)} \to \upsilon^{(t)})$
	sum to one over all $\upsilon^{(t)} \in \GT_t$.
\end{definition}


\begin{definition}
\label{def:pr_horizontal_transitions}
Let the coefficients $\mathrm{pr}\bigl(\upsilon^{(t)}\to\lambda^{(t-1)}\bigr)$ be defined through the following expansion into the linear basis of Schur functions:
\begin{equation}
\label{eq:Schur-branching-pr-coef}
\frac{s_{\upsilon^{(t)}} (x_1, \dots, x_{t-1}, 1)}{s_{\upsilon^{(t)}} (1^{t})}
=
\sum_{\lambda^{(t-1)} \in \GT_{t-1}}
\mathrm{pr}\bigl(\upsilon^{(t)} \to \lambda^{(t-1)}\bigr) \ssp
\frac{s_{\lambda^{(t-1)}} (x_1, \dots, x_{t-1})}{s_{\lambda^{(t-1)}} (1^{t-1})}.
\end{equation}
It follows from the branching rule for Schur functions that these coefficients have the explicit form
\begin{equation} \label{pr}
\mathrm{pr}\bigl(\upsilon^{(t)}\to\lambda^{(t-1)}\bigr)=
\begin{cases}
\dfrac{s_{\lambda^{(t-1)}}(1^{t-1})}{s_{\upsilon^{(t)}}(1^{t})}, & \lambda^{(t-1)}\prec \upsilon^{(t)},\\[6pt]
0, & \text{otherwise}.
\end{cases}
\end{equation}
By setting $x_i\equiv1$ in \eqref{eq:Schur-branching-pr-coef}, we see that the coefficients $\mathrm{pr}\bigl(\upsilon^{(t)}\to\lambda^{(t-1)}\bigr)$ sum to one over all $\lambda^{(t-1)}\in\GT_{t-1}$.
\end{definition}

\begin{definition}
\label{def:Aztec-signatures-measure}
	Given a sequence of parameters $(\beta_1, \dots,
	\beta_M)$, $0 < \beta_i <1$, define the probability
	measure $\mathbf{P}_{\beta_1, \dots, \beta_M}$ on the sequence
	of signatures of the form
	\begin{equation}
	\label{eq:Aztec-signatures-sequence-def}
		(\lambda^{(M)},
		\upsilon^{(M)}, \lambda^{(M-1)}, \upsilon^{(M-1)}, \dots, \lambda^{(2)}, \upsilon^{(2)},
		\lambda^{(1)}, \upsilon^{(1)})
	\end{equation}
	(we also set $\lambda^{(0)} = \varnothing$) by the formula
	\begin{multline}
		\label{eq:meas-Aztec}
\mathbf{P}_{\beta_1, \dots, \beta_M}\left(\lambda^{(M)}, \upsilon^{(M)}, \lambda^{(M-1)}, \upsilon^{(M-1)}, \dots, \lambda^{(2)}, \upsilon^{(2)}, \lambda^{(1)}, \upsilon^{(1)}\right)
		\\
		\coloneqq 1_{\lambda^{(M)} =
		(0^M) } \prod_{i=1}^{M} \kappa_{\beta_i} ( \lambda^{(i)} \to \upsilon^{(i)})
		\ssp
		\pr_{i \to (i-1)}(\upsilon^{(i)} \to \lambda^{(i-1)}).
	\end{multline}
	By induction, these weights sum to one over all sequences of signatures
	\eqref{eq:Aztec-signatures-sequence-def}.
\end{definition}
Let $\mathbb S_M$ be the set of
sequences \eqref{eq:Aztec-signatures-sequence-def} with nonzero
probability measure $\mathbf{P}_{\beta_1, \dots, \beta_M}$.
From explicit formulas \eqref{st} and \eqref{pr} it follows that each configuration from $\mathbb S_M$ has probability
\begin{equation*}
	\mathbf{P}_{\beta_1,\dots,\beta_M}\bigl(\lambda^{(M)},\upsilon^{(M)},\dots,\upsilon^{(2)},\lambda^{(1)}\bigr)
	=\prod_{i=1}^M(1-\beta_i)^{i}\prod_{i=1}^M\Bigl(\frac{\beta_i}{1-\beta_i}\Bigr)^{|\upsilon^{(i)}|-|\lambda^{(i)}|}.
\end{equation*}

The
measure
from \Cref{def:Aztec-signatures-measure}
corresponds to domino tilings of the Aztec diamond with one-periodic weights:
\begin{proposition}
	\label{prop:part-aztec-biject}
	There is a bijection
	between $\mathbb S_M$ and the set of domino tilings of the
	Aztec diamond of size $M$. Under this bijection,
	the measure
	$\mathbf{P}_{\beta_1, \dots, \beta_M}$
	\eqref{eq:meas-Aztec} turns into the measure
	on domino tilings with (fixed, deterministic) edge weights $W_i$
	as in \Cref{fig:AztecWeights-intro}, left,
	where $W_i=\beta_i/(1-\beta_i)$, $i=1,\ldots,M.$
\end{proposition}
\begin{proof}
This correspondence is classical; see, for example,
\cite{johansson2002non}, \cite{johansson2006eigenvalues}. A
detailed presentation in our notation is given in
\cite[Section~2]{BufetovKnizel2018}, where the same
construction is applied to domino tilings of more general
domains. \Cref{fig:AztecBijection} illustrates the
bijection for the Aztec diamond.
\end{proof}

\begin{figure}[htpb]
  \centering
	\includegraphics[width=0.8\textwidth]{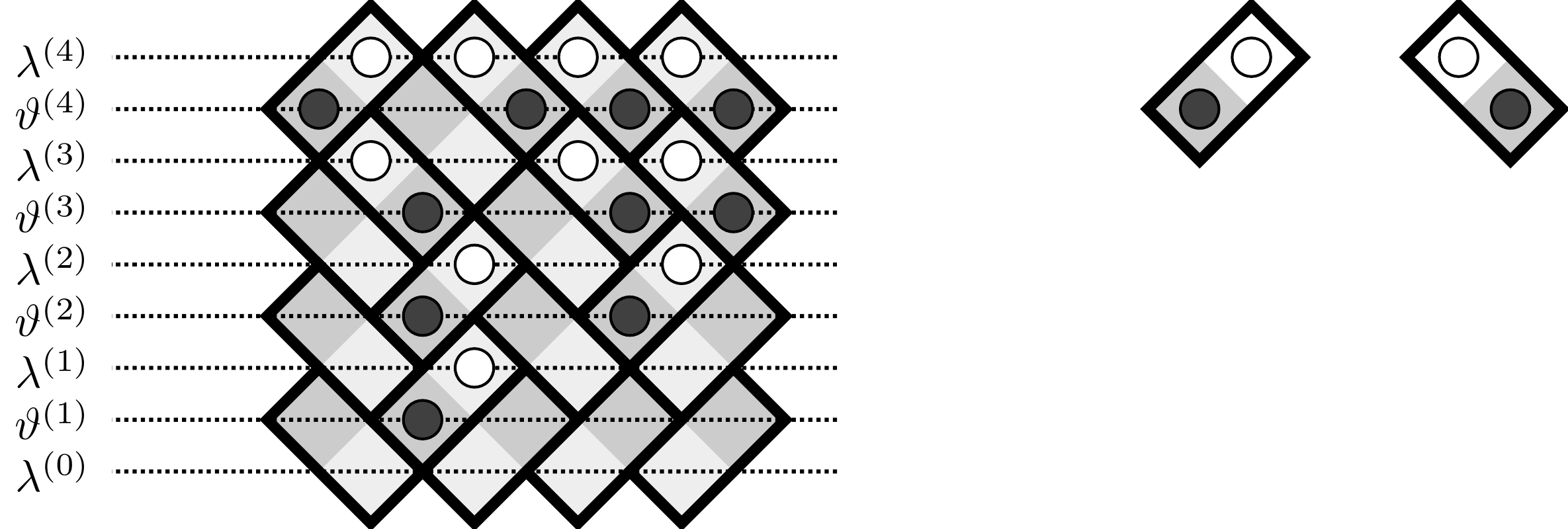}
	\caption{Correspondence between domino tilings of the
	Aztec diamond and sequences of signatures (here, $M=4$).
	We single out two types of dominos (red and green in the coloring in
	\Cref{fig:AztecWeights-intro}), and place black and white particles
	into them. The black (resp., white) particle configurations on each
	horizontal slice correspond to the signatures $\upsilon^{(i)}$
	(resp., $\lambda^{(i)}$). To read a signature, one
	counts the number of unoccupied positions
	to the left of each particle. In particular, we have
	$\lambda^{(4)}=(0,0,0,0)$,
	$\upsilon^{(4)}=(1,1,1,0)$,
	$\lambda^{(3)}=(1,1,0)$,
	$\upsilon^{(3)}=(2,2,1)$, and so on, until
	$\upsilon^{(1)}=(1)$ and $\lambda^{(0)}=(0)$.
	Note that $\sum_{i=1}^M
	|\upsilon^{(i)}| - |\lambda^{(i)}|$ is the number of
	NW-SE dominos in the corresponding row which contain particles.
	We assign
	nontrivial Boltzmann weights to these dominos.}
  \label{fig:AztecBijection}
\end{figure}

\Cref{prop:part-aztec-biject} allows to
translate results
about sequences of signatures
\eqref{eq:Aztec-signatures-sequence-def}
into the geometric
language of domino tilings of the Aztec diamond.
Throughout the paper, we mostly formulate and prove the results in the language of arrays and signatures,
as they are directly related to Schur generating functions described in the next
\Cref{subsec:SGF}.

\subsection{Schur generating functions}
\label{subsec:SGF}

Let $\rho_N$ be a probability measure on $\GT_N$.
A \emph{Schur generating function}
$S_{\rho_N} (x_1,\dots,x_N)$ is a symmetric Laurent (formal)
power series in $x_1,\dots,x_N$ defined by
\begin{equation*}
 S_{\rho_N} (x_1,\dots,x_N) \coloneqq \sum_{\lambda \in \GT_N} \rho_N (\lambda)
 \ssp
 \frac{s_\lambda(x_1,\dots,x_N)}{s_\lambda(1^N)}.
\end{equation*}
In what follows, we always assume that the measure $\rho_N$ is such that this
formal series is well-defined, i.e., the sum converges
in an open complex
neighborhood of $(1^N)$.

\begin{proposition}
\label{prop:SGF-comb}

Let $\mathbf{P}_{\beta_1, \dots, \beta_M}$ be a probability measure on sequences of signatures
\eqref{eq:Aztec-signatures-sequence-def} as in \Cref{def:Aztec-signatures-measure}.
Let $\rh_N$ be a projection of this measure to the signature $\lambda^{(N)}$, where $1 \le N \le M$. Then
this measuse admits the following Schur generating function:
\begin{equation*}
S_{\rh_N} (x_1,\dots,x_N) = \prod_{i=1}^N \prod_{j=N+1}^M \bigl(1 - \beta_j + x_i \ssp \beta_j\bigr).
\end{equation*}
\end{proposition}
\begin{proof}
This is a standard computation with Schur processes, a class of measures on sequences of signatures
introduced in \cite{okounkov2003correlation} (see also \cite{Borodin2010Schur}, \cite{Betea_etal2014}).
Namely, by
summing
\eqref{eq:meas-Aztec}
over all signatures except for $\lambda^{(N)}$ and using
usual and dual Cauchy identities, we obtain
\begin{equation}
\label{eq:Aztec-Schur-measure-of-lambda_N}
\rh_N \left( \lambda^{(N)} = \lambda \right) =
\Biggl(\prod_{i=N+1}^M \left( 1-\beta_i \right)^{N}\Biggr)\cdot
s_{\lambda} (1^N) \ssp
s_{\lambda'} \left( \frac{\beta_M}{1-\beta_M}, \frac{\beta_{M-1}}{1-\beta_{M-1}}, \dots, \frac{\beta_{N+1}}{1-\beta_{N+1}} \right).
\end{equation}
In particular, the random signature $\lambda^{(N)}$ is distributed according to the Schur measure
\cite{okounkov2001infinite}.
Applying the Cauchy identity once again to \eqref{eq:Aztec-Schur-measure-of-lambda_N},
one obtains the desired Schur generating function.
\end{proof}


\subsection{Asymptotics via Schur generating functions. Known results}
\label{subsec:Schur-generating-functions-known-results}

Here, we summarize the main results of
\cite{GorinBufetov2013free} and
\cite{bufetov2016fluctuations} that describe the asymptotic
behavior of random particle configurations on $\mathbb{Z}$
via Schur generating functions.

Let $\rho_N$ be a sequence of probability measures on $\GT_N$,
$N=1,2,\ldots$.
The main object of interest for us is the
asymptotic behavior of the following \textit{random
measure}:
\begin{equation}
	\label{eq:random-measure-m-rho-definition}
	m [ \rho_N ] \coloneqq \frac{1}{N} \sum_{i = 1}^N \delta \left(
	\frac{\lambda_i + N - i}{N} \right), \qquad \mbox{where $\lambda=(\lambda_1, \dots, \lambda_N)$ is $\rho_N$-distributed.}
\end{equation}
We study it via its moments:
\begin{equation}
	\label{eq:p_k-rho-moments-in-text}
	p_k^{(\rho_N)} \coloneqq \sum_{i=1}^N \left( \lambda_i +N-i
	\right)^k, \qquad \mbox{where $\lambda=(\lambda_1, \dots,
	\lambda_N)$ is $\rho_N$-distributed.}
\end{equation}

\begin{theorem}[Law of Large Numbers; Theorem 5.1 of \cite{GorinBufetov2013free}]
\label{theorem_moment_convergence}
 Suppose that the sequence $\rho_N$ is such that for every $k$ one has
\begin{equation}
\label{eq:assumptioN_BG_on_LLN}
 \lim_{N\to\infty} \frac{1}{N} \log \left( S_{\rho_N} (x_1,\dots,x_k, 1^{N-k}) \right) = \mathsf{F}(x_1)+\dots+ \mathsf{F}(x_k),
\end{equation}
where $\mathsf{F}$ is an analytic function in a complex
neighborhood of $1$,
and the convergence is uniform in an
open complex neighborhood of $1^k$.
Then the random measures $m [\rho_N]$ converge as
$N\to\infty$ in probability, in the sense of moments to a \emph{deterministic} measure $\mes$ on
$\mathbb R$, whose moments are given  by the following integral formula:
\begin{equation}
\label{eq_limit_moments_A}
 \int_{\mathbb R} x^k \ssp\mes(dx)= \frac{1}{2 \pi \ii (k+1)} \oint_{|z|= \varepsilon} \frac{dz}{1+z} \left( \frac{1}{z} +1 + (1+z) \ssp
 \mathsf{F} (1+z) \right)^{k+1}.
\end{equation}
\end{theorem}

\begin{theorem}[Central Limit Theorem; Theorem 2.8 of \cite{bufetov2016fluctuations}]
\label{th:BG2-one-level}
Assume that
$$
\lim_{N \to \infty} \frac{ \pa_1 \log S_{\rho_N} (x_1, \dots, x_k, 1^{N-k})}{N} = \mathsf{F} (x_1), \qquad \mbox{for any $\ k \ge 1$},
$$
$$
\lim_{N \to \infty} \pa_1 \pa_2 \log S_{\rho_N} (x_1, \dots, x_k, 1^{N-k})
= \mathsf{G} (x_1, x_2), \qquad \mbox{for any $\ k \ge 1$},
$$
where $\mathsf{F}(x), \mathsf{G}(x,y)$ are holomorphic functions, and the convergence is uniform in a complex neighborhood of unity.
Here, $\pa_1$ and $\pa_2$ denote the partial derivatives with respect to the first and second variables, respectively.

Then the collection of random variables
\begin{equation*}
	\left( \frac{ p_k^{(\rho_N)} - \raisebox{-1pt}{$\operatorname{\mathbf{E}}$} [ p_k^{(\rho_N)} ] \mathstrut}{N^{k } } \right)_{k \ge 1}
\end{equation*}
converges, as $N \to \infty$, in the sense of moments, to a Gaussian vector with zero mean and covariance
\begin{multline*}
\lim_{N \to \infty} \frac{\operatorname{Cov}(p_{k_1}^{(\rho_N)},
p_{k_2}^{(\rho_N)})}{N^{k_1+k_2}}
=
\frac{1}{(2 \pi \ii)^2}
\oint_{|z|=\varepsilon} \oint_{|w|=2 \varepsilon} \left( \frac{1}{z} +1 +
(1+z) \ssp\mathsf{F}(1+z) \right)^{k_1} \\ \times \left( \frac{1}{w} +1
+ (1+w) \ssp\mathsf{F}(1+w) \right)^{k_2} \left( \frac{1}{(z-w)^2} +
\ssp\mathsf{G}(1+z,1+w) \right) dz\ssp dw,
\end{multline*}
where the integration contours are counterclockwise, and $\varepsilon \ll 1$.
\end{theorem}
We also
note that \cite[Section~2.4]{bufetov2016fluctuations}
provides multilevel generalizations of
the above Central Limit Theorem, which concern limits of covariances
of the form
$N^{-k_1-k_2}\ssp\operatorname{Cov}\bigl(p_{k_1}^{(\rho_{\lfloor \alpha_1 N \rfloor})},\,p_{k_2}^{(\rho_{\lfloor \alpha_2 N \rfloor})}\bigr)$.

\begin{remark}
	Observe that after exponentiation, the left-hand side of \eqref{eq:assumptioN_BG_on_LLN} becomes the same as
	the expression in
	our \Cref{th:LLN-intro} from the Introduction.
	However, this simple correspondence does not extend to the level of the
	Central Limit Theorem (compare to \Cref{th:CLT-gen} below).
	Indeed, the CLT computations require differentiating either the
	logarithm or the $N$-root of the Schur generating function,
	which, in general, yields different results.
\end{remark}

\section{Domino tilings of the Aztec diamond with random edge weights}
\label{sec:model}

Let us now describe the model we study in the present paper.
We consider domino tilings of the Aztec
diamond of size $M$ with random edge weights (illustrated in \Cref{fig:AztecWeights-intro}).
Equivalently, via the
encoding of \Cref{subsec:comb},
we can (and will) work with sequences of signatures
$(\lambda^{(N)}, \upsilon^{(N)}, \lambda^{(N-1)},
\upsilon^{(N-1)}, \dots, \lambda^{(2)}, \upsilon^{(2)},
\lambda^{(1)}, \upsilon^{(1)})$
\eqref{eq:Aztec-signatures-sequence-def}
and the measure
$\mathbf{P}_{\beta_1, \dots, \beta_M}$ on them (\Cref{def:Aztec-signatures-measure}),
but now we also choose the parameters
$(\beta_1, \dots, \beta_M)$ to be random. We will use the
notation $(\bb_1, \bb_2, \dots, \bb_M)$ for these parameters in order to
emphasize that from now on these are \textbf{random variables} rather
than fixed reals.

\begin{definition}[Random environment with i.i.d.\ weights]
\label{def:iid_environment}
	We consider the case where $\{\bb_i\}_{i=1}^M$ are i.i.d.\
	random variables drawn from a distribution $\BB_M$, which
	may depend on $M$. We assume that the distribution
	$\BB_M$
	has finite moments of all orders.
\end{definition}

We will study two regimes: one in which the
distribution is fixed, and another in which it varies with
$M$ so that the fluctuations due to the random weights and
those due to the domino tiling randomness occur on the same
scale.

\medskip
In the context of random environments, it is common to
distinguish between two types of expectations: the
\textbf{quenched} expectation, where we first fix the random
environment (in our case, the parameters
$\vec{\bb} = (\bb_1, \ldots, \bb_M)$) and then
compute expectations with respect to the domino tiling, and
the \textbf{annealed} expectation, where we
average over both
the randomness in the tiling and the randomness in the
environment.
Let us denote by $\operatorname{\mathbf{E}}_{\BB}$
the expectation with respect to the randomness coming from
the environment (i.e., the distribution of $\vec{\bb}$).
This notation is not strictly necessary (we might as well
write the usual expectation $\operatorname{\mathbf{E}}$),
but it helps emphasize the randomness coming from the environment.

For the i.i.d.\ environment $\vec{\bb} = (\bb_1, \ldots, \bb_M)$,
the annealed
Schur generating function has the following explicit form:

\begin{proposition}
\label{prop:SGF-rand}
In the sequence of signatures
$(\lambda^{(N)}, \upsilon^{(N)},
\dots,
\lambda^{(1)}, \upsilon^{(1)})$
\eqref{eq:Aztec-signatures-sequence-def}
under the measure
$\mathbf{P}_{\bb_1, \bb_2 \dots, \beta_M}$ (with $\bb_i$ random
as in \Cref{def:iid_environment}), let $\rho_N$ be the marginal
distribution of $\lambda^{(N)}$. Then its annealed Schur generating function is given by
\begin{equation}
	\label{eq:annealed_Schur_gf_main_statement}
	S_{\mathbf{\rho}_N} (x_1,\dots,x_N) = \left( \operatorname{\mathbf{E}}_{\BB} \prod_{i=1}^N \left( 1 - \bb + x_i \bb \right) \right)^{M-N}.
\end{equation}
\end{proposition}
\begin{proof}
Applying \Cref{prop:SGF-comb}, we have
\begin{equation*}
	S_{\mathbf{\rho}_N} (x_1,\dots,x_N)
	=
	\operatorname{\mathbf{E}}_{\BB}
	\prod_{i=1}^N \prod_{j=N+1}^M \left( 1 - \bb_j+ x_i\ssp \bb_j \right),
\end{equation*}
which simplifies to the right-hand side of \eqref{eq:annealed_Schur_gf_main_statement}
thanks to the independence of the environment random variables
$\bb_j$.
\end{proof}

\begin{remark}
\label{rem:SGF-rand}
Although the proof of \Cref{prop:SGF-rand} is very
short, it plays an important conceptual role. Domino
tilings of the Aztec diamond is a well-studied model,
which is amenable to asymptotic analysis by a
number of approaches.
These include
analysis of the shuffling algorithm \cite{jockusch1998random}
and further application of cluster algebras
\cite{diFrancesco2014arctic_via_cluster},
variational principle \cite{CohnKenyonPropp2000},
\cite{OkounkovKenyon2007Limit},
determinantal point process methods
\cite{johansson2002non}, \cite{johansson2006eigenvalues},
\cite{chhita2015asymptotic},
the use of orthogonal polynomials
\cite{duits2017two},
(nonrigorous) tangent method
\cite{DiFrancescoLapa2018},
Schur generating functions
\cite{GorinBufetov2013free}, \cite{bufetov2016fluctuations},
\cite{BufetovKnizel2018},
and
graph embeddings focused on conformal structure
\cite{chelkak2024fluctuations}, \cite{berggren2024perfect}.

\Cref{prop:SGF-rand} shows that the
method of Schur generating functions is particularly
suited to analyzing random weights in
one-periodic situations (equivalently, random
parameters of Schur measures and processes).
All other methods listed above
seem to encounter immediate technical obstacles
in the random environment setting.

We also note that Schur generating functions,
while suitable for studying the global behavior of
domino tilings (even in one-periodic random environment),
are not well-suited for
their local behavior. Thus, understanding the local
behavior of domino tilings in a random environment
requires overcoming these technical obstacles.
\end{remark}

\section{Decreasing variance: Gaussian Free Field plus Brownian motion}
\label{sec:decreasing-variance}

\subsection{Law of Large Numbers and Central Limit Theorem}
\label{sub:LLN-CLT-decreasing-variance}

Here we
study
domino tilings with random edge weights (\Cref{sec:model})
under the assumption that the distribution $\BB=\BB_M$ depends on the size $M$ of the Aztec diamond,
has compact support on $\R_{>0}$ uniformly in $M$, and
its first two moments satisfy
\begin{equation}
\label{eq:assum1}
\lim_{M \to \infty} \operatorname{\mathbf{E}} [\bb] = \beta,
\qquad \lim_{M \to \infty} M\cdot \operatorname{Var}[\bb] = \sigma^2,
\qquad \bb \sim \BB_M,
\end{equation}
where $0<\beta<1$ and $\sigma>0$.
Note that \eqref{eq:assum1}
implies that as $M \to \infty$, $\bb$ convreges in probability to the deterministic value $\beta$.
Let us also pick $0<\alpha<1$, and assume that the index $N$ of the random signature $\lambda^{(N)}$ in the sequence of signatures \eqref{eq:Aztec-signatures-sequence-def} (corresponding to a domino tiling of the Aztec diamond of size $M$) satisfies
\begin{equation}
\label{eq:assum1_prim}
\lim_{M \to \infty} \frac{N}{M} = \alpha.
\end{equation}

\begin{remark}
	The results of this section do not require the enhanced asymptotics via Schur generating functions formulated
	in \Cref{subsec:Schur-generating-functions-intro}, and are obtained as a combination of
	\Cref{prop:SGF-rand} and the known results of \cite{GorinBufetov2013free}, \cite{bufetov2016fluctuations}.
\end{remark}

\begin{theorem}
\label{thm:limit_shape_decreasing_variance}
Under assumptions \eqref{eq:assum1}--\eqref{eq:assum1_prim},
let
$h_M$ be the (unrescaled) domino height
function of an Aztec diamond of size $M$ with edge weights $\bb_i$.
Then the rescaled height profile
$M^{-1} h_M\bigl(\lfloor Mu\rfloor,\lfloor Mv\rfloor\bigr)$
converges (in probability, in the sense of moments of the measures $m[\rho_N]$
\eqref{eq:random-measure-m-rho-definition}
for any $N\sim \alpha M$, $0<\alpha<1$)
to a deterministic limit shape which is the same as
in the case of nonrandom parameters $\bb_i\equiv \beta$ for all $i$.
\end{theorem}
\begin{proof}
By \Cref{prop:SGF-rand} , the
normalized logarithm of the annealed Schur generating function equals
\begin{equation}
	\label{eq:annealed_Schur_gf_log_proof}
	\frac{1}{N}\log \ssp S_{\mathbf{\rho}_N} (x_1,\dots,x_k,1^{N-k})
	=
	\frac{M-N}{N}\ssp \log\ssp
	\operatorname{\mathbf{E}}_{\BB} \prod_{i=1}^k \left( 1 - \bb + x_i \bb \right).
\end{equation}
Since $\bb$ converges in probability to $\beta$,
the expectation factorizes in the limit.
Thus, the right-hand side of \eqref{eq:annealed_Schur_gf_log_proof}
converges to the same expression $(\alpha^{-1}-1)\sum_{i=1}^{k}\log(1-\beta+x_i\ssp \beta)$
as for the deterministic parameters $\bb_i\equiv \beta$.
By
\Cref{theorem_moment_convergence}, we get the desired result.
\end{proof}

Next, let us compute the asymptotic quantities entering \Cref{th:BG2-one-level}:

\begin{proposition}
\label{prop:SGF-asym-old}
Under assumptions \eqref{eq:assum1}--\eqref{eq:assum1_prim}, for any $k \in \mathbb{Z}_{\ge1}$, we have
\begin{equation*}
\lim_{M \to \infty} \frac{1}{N}\ssp \pa_1 \log \Bigl( \operatorname{\mathbf{E}}_{\BB} \prod_{i=1}^k \bigl( 1 - \bb + x_i\ssp \bb \bigr) \Bigr)^{M-N} = \Bigl( \frac{1}{\alpha} - 1 \Bigr) \frac{\beta}{1 - \beta + \beta x_1},
\end{equation*}
and
\begin{equation*}
\lim_{M \to \infty} \pa_1 \pa_2 \log \Bigl( \operatorname{\mathbf{E}}_{\BB} \prod_{i=1}^k \bigl( 1 - \bb + x_i\ssp \bb \bigr) \Bigr)^{M-N} = \frac{(1-\alpha)\ssp \sigma^2}{(1 - \beta + \beta x_1)^2 (1 - \beta + \beta x_2)^2}.
\end{equation*}
\end{proposition}

\begin{proof}
Differentiating with respect to $x_1$, we have
$$
\pa_1 \log \left( \operatorname{\mathbf{E}}_{\BB} \prod_{i=1}^k \left( 1 - \bb + x_i\ssp\bb \right) \right)
=
\frac{\operatorname{\mathbf{E}}_{\BB} \left[ \bb \prod_{i=2}^k \left( 1 - \bb + x_i\ssp\bb \right) \right] }{\operatorname{\mathbf{E}}_{\BB} \left[ \prod_{i=1}^k \left( 1 - \bb + x_i\ssp\bb \right) \right] }.
$$
Since $\bb$ converges to a deterministic value $\beta$ as $M \to \infty$, we immediately
get the first statement of the proposition.

Further differentiating with respect to $x_2$, we have
\begin{multline*}
\pa_2 \pa_1 \log \left( \operatorname{\mathbf{E}}_{\BB} \prod_{i=1}^k \left( 1 - \bb + x_i\ssp\bb \right) \right)
=
\frac{ \operatorname{\mathbf{E}}_{\BB} \left[ \bb^2 \prod\limits_{i=3}^k \left( 1 - \bb + x_i\ssp\bb \right) \right] \operatorname{\mathbf{E}}_{\BB} \left[ \prod\limits_{i=1}^k \left( 1 - \bb + x_i\ssp\bb \right) \right]}{\operatorname{\mathbf{E}}_{\BB} \left[ \prod\limits_{i=1}^k \left( 1 - \bb + x_i\ssp\bb \right) \right]^2}
\\ -\frac{\operatorname{\mathbf{E}}_{\BB} \left[ \bb \prod\limits_{i=2}^k \left( 1 - \bb + x_i\ssp\bb \right) \right] \operatorname{\mathbf{E}}_{\BB} \left[ \bb \prod\limits_{i=1; i \ne 2}^k \left( 1 - \bb + x_i\ssp\bb \right) \right] }
{ \operatorname{\mathbf{E}}_{\BB} \left[ \prod\limits_{i=1}^k \left( 1 - \bb + x_i\ssp\bb \right) \right]^2 }.
\end{multline*}
Set
$$
\xi\coloneqq \bb - \beta,
\qquad
\psi\coloneqq \prod_{i=3}^k ( 1 - \bb + \bb x_i ).
$$
Next, the numerator of the expression above,
\begin{multline*}
\operatorname{\mathbf{E}}_{\BB} \left[ \bb^2 \psi \right]
\operatorname{\mathbf{E}}_{\BB} \left[ \left( 1 - \bb + x_1
\bb \right) \left( 1 - \bb + x_2 \bb \right) \psi \right]
\\
-
\operatorname{\mathbf{E}}_{\BB} \left[ \left( 1 - \bb + x_1
\bb \right) \bb \psi \right] \operatorname{\mathbf{E}}_{\BB}
\left[ \left( 1 - \bb + x_2 \bb \right) \bb \psi \right],
\end{multline*}
can be expanded into a Taylor series with terms of order 1, $\xi$, $\xi^2$, and $O(\xi^3)$. After a direct calculation, one sees that the leading term as $M\to\infty$ is
$$
\left (\operatorname{\mathbf{E}}_{\BB} ( \xi^2) -
\operatorname{\mathbf{E}}_{\BB} ( \xi)^2 \right)
\operatorname{\mathbf{E}}_{\BB} [\psi^2]
\approx \frac{\sigma^2}{M} \operatorname{\mathbf{E}}_{\BB} [\psi^2].
$$
Taking into account the denominator and using the convergence of $\bb$ to $\beta$, we obtain the
second statement of the proposition.
\end{proof}

Recall that $\rho_N$ denotes the distribution of $\lambda^{(N)}$,
and its moments $p_k^{(\rho_N)}$ are defined in \eqref{eq:p_k-rho-moments-in-text}.
The next theorem is the first main result of this paper.

\begin{theorem}
\label{th:main-decreasing}
Under assumptions \eqref{eq:assum1}--\eqref{eq:assum1_prim},
the normalized moments
$M^{-k} p_{k}^{(\rho_N)}$, $k\in \mathbb{Z}_{\ge1}$, are jointly asymptotically Gaussian, with the limiting covariance
given by
\begin{multline}
\label{eq:main-decreasing}
\lim_{N \to \infty} \frac{\operatorname{Cov}\bigl(p_{k_1}^{(\rho_N)},\,p_{k_2}^{(\rho_N)}\bigr)}{M^{k_1+k_2}}
= \frac{ \alpha^{k_1+k_2}}{(2\pi \mathbf{i})^2}
	\oint_{|z|=\varepsilon}
	\oint_{|w|=2\varepsilon}
	\left( \frac{1}{z}+1+\frac{(1+z)(1-\alpha)\beta}{\alpha\ssp (1-\beta+\beta(z+1))} \right)^{k_1}
	\\\times
	\left( \frac{1}{w}+1+\frac{(1+w)(1-\alpha)\beta}{\alpha\ssp (1-\beta+\beta(w+1))} \right)^{k_2}
	\left(
		\frac{(1-\alpha)\ssp\sigma^2}{(1+\beta z)^2(1+\beta w)^2}
		+\frac{1}{(z-w)^{2}}
	\right)
	\,dz\,dw.
\end{multline}
The integration contours are counterclockwise and $\varepsilon \ll 1$.
\end{theorem}
\begin{proof}
This is a straightforward combination of \Cref{th:BG2-one-level} and \Cref{prop:SGF-asym-old}.
\end{proof}

One also has a multilevel version of \Cref{th:main-decreasing}:
\begin{theorem}
\label{th:main-decreasing-multilevel}
Let $k_1, k_2$ be two integers, and let $N_1 \le N_2$ be two sequences such that $N_1 / M \to \alpha_1$, $N_2 / M \to \alpha_2$, for $\alpha_1 < \alpha_2$. Then, the family of random variables
$\{ M^{-k} p_{k}^{(\rho_{N_1})}, M^{-l} p_{l}^{(\rho_{N_2})} \}_{k,l \in \mathbb{Z}_{\ge1}}$ is asymptotically Gaussian,
with the limiting covariance given by
\begin{multline}
\label{eq:main-decreasing-multilevel}
\lim_{N \to \infty} \frac{\operatorname{Cov}\bigl(p_{k_1}^{(\rho_{N_1})},\,p_{k_2}^{(\rho_{N_2})}\bigr)}{M^{k_1+k_2}}
= \frac{ \alpha_1^{k_1} \alpha_2^{k_2} }{(2\pi \mathbf{i})^2}
	\oint_{|z|=\varepsilon}
	\oint_{|w|=2\varepsilon}
	\left( \frac{1}{z}+1+\frac{(1+z)(1-\alpha_1)\ssp\beta}{\alpha_1\ssp (1-\beta+\beta(z+1))} \right)^{k_1}
	\\\times
	\left( \frac{1}{w}+1+\frac{(1+w)(1-\alpha_2)\ssp\beta}{\alpha_2\ssp (1-\beta+\beta(w+1))} \right)^{k_2}
	\left(
		\frac{(1-\alpha_2)\ssp \sigma^2}{(1+\beta z)^2(1+\beta w)^2}
		+\frac{1}{(z-w)^{2}}
	\right)
	\,dz\,dw.
\end{multline}
\end{theorem}
\begin{proof}
The extension of \Cref{th:main-decreasing} to the multilevel
setting follows the same steps as the passage from one-level
to multilevel results in \cite{bufetov2016fluctuations}.
That paper introduces multilevel Schur generating functions;
for the Schur process these functions can be evaluated
explicitly, and with random parameters the computation is
analogous to \Cref{prop:SGF-rand}. Finally, instead of
\Cref{th:BG2-one-level} we invoke its multilevel analogue,
\cite[Theorem 2.11]{bufetov2016fluctuations}, which applies
to multilevel distributions on signatures.
We omit further details of the proof.
\end{proof}

\subsection{Limiting covariance: Gaussian Free Field plus Brownian motion}
\label{sub:covariance-structure-decreasing-variance}

Let us interpret the covariance structure \eqref{eq:main-decreasing-multilevel} obtained in \Cref{th:main-decreasing-multilevel}.
We can represent the integral as a sum of two terms, corresponding to the two summands
\begin{equation}
\label{eq:covariance-structure-decreasing-variance-two-terms}
		\frac{(1-\alpha_2)\ssp \sigma^2}{(1+\beta z)^2(1+\beta w)^2}
		+\frac{1}{(z-w)^{2}}
\end{equation}
in one of the brackets.
The second term already appears in the case of deterministic
edge weights. As shown in \cite[Section~9]{bufetov2016fluctuations}, this term gives rise to the
\emph{Gaussian Free Field} (GFF). That is, in an appropriate complex
structure determined by the limit shape, the contour
integral expression can be transformed into an expression
coming from the GFF.

In more detail, if we map the liquid region inside the Aztec diamond
(i.e., the region bounded by the arctic ellipse inside of which one sees a random mixture of dominos of all types,
cf.~\Cref{fig:uniform_200})
to the complex upper half-plane $\mathbb{H}\coloneqq \{ z \in \mathbb{C} : \Im(z) > 0 \}$
via a suitable diffeomorphism,\footnote{Which is 
deterministic, that is, does not depend on our random edge weights.}
then
the covariance of fluctuations at points $z,w \in
\mathbb{H}$ will be given by
$$
\frac{-1}{2 \pi^2} \log \left( \frac{z-w}{z - \overline{w}} \right),
$$
which is a scalar multiple of the Green function of the Laplace operator in $\mathbb{H}$ with Dirichlet boundary conditions.

The term with the first summand in \eqref{eq:covariance-structure-decreasing-variance-two-terms} is new.
It
appears because of the randomness coming from the edge weights.
Under the same diffeomorphism of the liquid region onto $\mathbb{H}$,
we obtain the following covariance term in the GFF complex structure:
$$
\frac{4 \Im (z) \Im (w) \min \left( 1- \alpha_1, 1-\alpha_2 \right) }{ |1 + \beta z|^2\ssp |1+\beta w|^2},\qquad
z,w \in \mathbb{H}.
$$
One can verify that this fluctuation term is interpreted as
coming from
(a deterministic pushforward of) the one-dimensional Brownian motion.
Note that in \Cref{th:main-decreasing-multilevel} we assumed that $\alpha_1<\alpha_2$,
but since the left-hand side of \eqref{eq:main-decreasing-multilevel} is symmetric in $\alpha_1$ and $\alpha_2$,
we can write
$1-\alpha_2=\min \left( 1- \alpha_1, 1-\alpha_2 \right)$ in order to emphasize the Brownian motion-like covariance structure.

To summarize, it is natural to interpret the additional
Brownian motion appearing in the covariance as arising from
the limiting fluctuations of the i.i.d.\ parameters
$(\bb_1,\dots,\bb_M)$. At the same time, tuning the variance of the edge weights
as $\sim 1/M$, the fluctuation structure still retains the Gaussian Free Field
component, which is the same as in the case of deterministic edge weights.

\section{Law of Large Numbers}
\label{sec:LLN}

In this section we prove the law of large numbers for
discrete particle systems determined by their Schur
generating functions. Our result vastly generalizes
the one from \cite{GorinBufetov2013free}
(formulated as \Cref{theorem_moment_convergence}
above).
We require this
generalization in order to analyze domino tilings of the Aztec
diamond with random edge weights of a fixed distribution
(as opposed to the decreasing variance case in \Cref{sec:decreasing-variance}).
However, this law of large numbers is certainly of an independent interest.

For the proof in this and the next section, we use the
method of Schur generating functions that was developed in \cite{GorinBufetov2013free},
\cite{bufetov2016fluctuations}. While we recall all the
necessary steps, we focus primarily on new ideas and
computations. We refer to \cite{bufetov2016fluctuations} and
\cite{BufetovZografos2024} for a more detailed discussion of
the basics of the method. Preliminary definitions
and results on Schur generating functions from \cite{GorinBufetov2013free}, \cite{bufetov2016fluctuations}
are given in \Cref{subsec:SGF,subsec:Schur-generating-functions-known-results} above.

\subsection{Symmetrization lemmas}

First, we recall a symmetrization result from \cite{GorinBufetov2013free}.
Define $P \coloneqq \{ (a, b) \in
\mathbb{Z}_{\ge 1}^2 \of 1 \leq a < b \leq n \}$.
\begin{lemma}[{\cite[Lemma~5.4]{GorinBufetov2013free}}]
  \label{lem:lemma1} Let $n$ be a positive integer. Let $f (z_1, \ldots,
  z_n)$ be a function. We consider its symmetrization with respect to $P$:
  \[ f_P (z_1, \ldots, z_n) \coloneqq \frac{1}{n!} \sum_{\pi\in S_n}
	\frac{f (z_{\pi (1)}, \ldots, z_{\pi (n)})}{\prod_{(a, b) \in P} (z_{\pi_{(a)}} -
     z_{\pi_{(b)}})}  \]
	(where $S_n$ is the symmetric group).
  Then:
\begin{enumerate}[$\bullet$]
\item If $f$ is analytic in a complex neighborhood of $1^n$,
then $f_P$ is also analytic in a complex neighborhood of
$1^n$.
\item If $(f^{(m)}(z_1,\ldots,z_n))_{m\ge1}$ is a
sequence of analytic functions converging to zero uniformly
in a complex neighborhood of $1^n$, then so is the sequence
$(f_P^{(m)})_{m\ge1}$.
\end{enumerate}
\end{lemma}

We need a new symmetrization identity 
which we were not able to locate in the literature.
Throughout the rest of the paper, we will use the notation
\begin{equation*}
	w_m \coloneqq \exp \left( \frac{2 \pi \sqrt{-1}}{m} \right)
\end{equation*}
for the $m$-th root of unity.

\begin{lemma}
\label{lem:main-sym}
	Let $G(u_1, u_2, \dots, u_{m+1})$ be a complex-analytic
	function in a neighborhood of $(0^{m+1})$, which is
	symmetric in $u_2, \dots, u_{m+1}$ (but not necessarily in
	$u_1$). We have
	\begin{equation}
	\label{eq:basic-sym}
\sum_{i=1}^{m+1} \frac{ G( u_i, u_1, u_2, \dots, \widehat{u_i}, \dots, u_{m+1} )}
 { \prod_{j=1 ;\ssp j \neq i}^{m+1} (u_i - u_j) } \bigg\vert_{u_1=\dots =u_{m+1}=0}
 =
 \frac{1}{m!}\ssp \pa_u^m \bigl[ G(u, u w_{m+1},
 \dots, u w_{m+1}^m) \bigr] \Big\vert_{u=0}.
	\end{equation}
\end{lemma}
\begin{proof}
	By \Cref{lem:lemma1}, the limit in the left-hand side
	of \eqref{eq:basic-sym} exists. In order to compute it, we
	chose $u_i = u w_{m+1}^{i-1}$, $i=1,2, \dots, m+1$, and
	let $u$ tend to zero. Define also
	$$
	\widetilde F (u)\coloneqq G(u, u w_{m+1}, \dots, u w_{m+1}^m).
	$$
	For our choice of $u_i$'s, one has
	$$
	\widetilde F (u_i) = G( u_i, u_1, u_2, \dots, \widehat{u_i}, \dots, u_{m+1} ),
	$$
	due to the symmetry of $G$ with respect to all variables except of the first one, and the fact that $w_{m+1}$ is a root of unity.
	Therefore,
	\begin{equation}
\label{eq:basic-sym-2}
\sum_{i=1}^{m+1} \frac{ G( u_i, u_1, u_2, \dots,
\widehat{u_i}, \dots, u_{m+1} )}
{ \prod_{j=1 ; j \neq
i}^{m+1} (u_i - u_j) } \bigg\vert_{u_1= \dots= u_{m+1}=0}
=
\sum_{i=1}^{m+1} \frac{ \widetilde F (u_i) }{ \prod_{j=1 ; j \neq i}^{m+1} (u_i - u_j) } \bigg\vert_{u_1= \dots= u_{m+1}=0}.
\end{equation}
	The symmetrization in the right-hand side of \eqref{eq:basic-sym-2} is
	now standard. It is well-known
	(see, e.g., \cite[Lemma~5.5]{GorinBufetov2013free})
	that this right-hand side is equal to $\frac{1}{m!} \left. \pa_u^m \widetilde F (u) \right\vert_{u=0}$. This completes the proof.
\end{proof}

The following corollary is obtained immediately from \Cref{lem:main-sym} via a shift of variables.
\begin{corollary}
Let $G(u_1, u_2, \dots, u_{m+1})$ be a complex-analytic
function in a neighborhood of $(1^{m+1})$, which is
symmetric in $u_2, \dots, u_{m+1}$ (but not necessarily in
$u_1$). Then
\begin{multline}
\label{eq:corol-sym}
\sum_{i=1}^{m+1}  \frac{ G( u_i, u_1, u_2, \dots, \hat u_i, \dots, u_{m+1} )}{ \prod_{j=1 ; j \neq i}^{m+1} (u_i - u_j) } \bigg\vert_{u_1= \dots= u_{m+1}=1}
\\ = \frac{1}{m!}\ssp \pa_u^m \left[ G(1+u, 1+u w_{m+1}, \dots, 1+u w_{m+1}^m) \right] \Big\vert_{u=0}.
\end{multline}

\end{corollary}

The next observation will be important for simplifying formulas in applications:
\begin{lemma}
\label{lem:roots-simplifying-sym}
Let $P(u_1, u_2, \dots, u_{m})$ be a constant-free polynomial of degree strictly less than $m$. Then
$P(1,w_m, w_m^2, \dots, w_{m}^{m-1})=0$.
\end{lemma}
\begin{proof}
Power sums of degree less than $m$ evaluated at
all $m$-th roots of unity vanish,
that is, $\sum_{i=0}^{m-1}w_m^{i\cdot k}=0$ for all $1\le k<m$.
Any polynomial $P$ as in the hypothesis can be expressed,
as a linear combination of products of such
power sums. This completes the proof.
\end{proof}

\subsection{Law of Large Numbers}
\label{subsec:LLN_general_setting}

Let $\rho_N$ be a probability measure on $\GT_N$. Recall that our main object of interest is the asymptotic behavior of the empirical measure
$m[\rho_N]$ defined by \eqref{eq:random-measure-m-rho-definition}.
The main technical role in our proofs is played (as in
\cite{GorinBufetov2013free}, \cite{bufetov2016fluctuations})
by differential operators acting on analytic functions of
$N$ variables,
\[ \bigl(\mathcal{D}_k f\bigr) (u_1, \ldots, u_N) \coloneqq \frac{1}{V_N (\vec{u})}
\sum_{i = 1}^N (u_i \partial_i)^k [V_N (\vec{u}) f (u_1, \ldots, u_N)], \]
where $V_N (\vec{u}) \coloneqq \prod_{1 \leq i < j \leq N} (u_i - u_j)$
is the Vandermonde determinant,
and $\partial_i$ denotes the derivative with respect to $u_i$.
They are useful because they act diagonally on the Schur functions:
\begin{proposition}
\label{prop:eigen}
  For $\lambda \in \GT_N$, we have
  \[
	\mathcal{D}_k\ssp s_{\lambda}
	=
	s_{\lambda}\cdot
	\sum_{i = 1}^N (\lambda_i + N - i)^k
   .
	\]
\end{proposition}
\begin{proof}
  Immediate from the definition of Schur functions \eqref{eq:Schur-function-def}.
\end{proof}

We now state and prove the main result of this section:
\begin{theorem}
  \label{th:LLN}
  Let $\rho_N$, $N \ge 1$, be a sequence of probability measures on $\GT_N$. Assume that there exists a sequence of symmetric functions $F_k(z_1,\ldots,z_k)$, $k \ge 1$, analytic in a complex neighborhood of $1^k$, such that for every fixed $k$ we have
  \begin{equation}
    \lim_{N \rightarrow \infty} \sqrt[N]{S_{\rho_N}(u_1,\ldots,u_k,1^{N-k})} = F_k(u_1,\ldots,u_k), \label{assumptioN1}
  \end{equation}
	where the convergence is uniform in a complex neighborhood
	of $1^k$. Then, as $N \rightarrow \infty$, the random
	measures $m[\rho_N]$ converge in probability, in the sense
	of moments, to a deterministic probability measure
	$\tmmathbf{\mu}$ on $\mathbb{R}$. The moments
	$(\tmmathbf{\mu}_k)_{k\ge1}$ of this measure are given by
  \begin{multline}
    \tmmathbf{\mu}_k = \sum_{l = 0}^k \binom{k}{l} \frac{1}{(l + 1)!}
    \\
    \times
    \pa_u^l \Bigl[(1+u)^k\bigl(\pa_1 F_{l+1}(1+u,1+u w_{l+1},1+u w_{l+1}^2,\dots,1+u w_{l+1}^l)\bigr)^{k-l}\Bigr]\Big\vert_{u=0}.
    \label{eq:LLN-gen}
  \end{multline}
\end{theorem}

\begin{proof}
\textbf{Step 1.}
First, we show convergence in expectation.
By \Cref{prop:eigen}, we have
$$
\left. \mathcal{D}_k\ssp S_{\rho_N} (u_1, u_2, \dots, u_N) \right|_{u_1=\dots= u_N=1}
=
\operatorname{\mathbf{E}}\ssp\biggl[
\sum_{i = 1}^N (\lambda_i + N - i)^k
\biggr],
$$
where $\lambda$ is distributed according to $\rho_N$. The expression
\begin{equation}
    \frac{1}{V_N (\vec{u})} \sum_{i = 1}^N u_i^n \partial_i^n [V_N (\vec{u})
    S_{\rho_N} (u_1, \ldots, u_N)] \label{sum}
  \end{equation}
can be written as
\begin{equation}
    \sum_{m = 0}^k\ssp
\sum_{\substack{l_0,\dots,l_m = 1 \\ l_i \neq l_j \text{ for } i \neq j}}^N
		\binom{k}{m} \frac{u_{l_0}^k \partial_{l_0}^{k - m}
    S_{\rho_N} (u_1, \ldots, u_N)}{(u_{l_0} - u_{l_1}) \ldots (u_{l_0} -
    u_{l_m})} . \label{sum2}
  \end{equation}
Further on, using the form $S_{\rho_N} = (\sqrt[N]{S_{\rho_N}})^N$ when applying the differential operator, one can write \eqref{sum2} as a linear combination of terms of the form
  \[ \frac{u_{b_0}^k \left( \partial_{b_0} \sqrt[N]{S_{\rho_N}} \right)^{l_1} \ldots \left( \partial_{b_0}^{k - m}
      \sqrt[N]{S_{\rho_N}} \right)^{l_{k -
     m}}}{(u_{b_0} - u_{b_1}) \ldots (u_{b_0} - u_{b_m})} +
     \text{{\hspace{13em}}} \]
  \begin{equation}
    \text{{\hspace{12em}}} \ldots + \frac{u_{b_m}^k \left( \partial_{b_m} \sqrt[N]{S_{\rho_N}} \right)^{l_1} \ldots \left( \partial_{b_m}^{k - m}
      \sqrt[N]{S_{\rho_N}} \right)^{l_{k -
     m}}}{(u_{b_m} - u_{b_0}) \ldots (u_{b_m} - u_{b_{m
    - 1}})}, \label{sum3}
  \end{equation}
  where $l_1 + 2 l_2 + \ldots + (k - m) l_{k - m} = k - m$, $b_0, \ldots, b_m
  \in \{ 1, \ldots, N \}$ are distinct, and $m = 0, \ldots, k$.

	By \Cref{lem:lemma1}, expression \eqref{sum3} has a limit
	as $u_1, \ldots, u_N \rightarrow  1$. Furthermore, the
	symmetry of $S_{\rho_N}$ guarantees that the limit does
	not
  depend on $b_0, \ldots, b_m \in \{ 1, \ldots, N \}$. Consequently, this
  limit will appear $m!\ssp \binom{N}{m + 1}$ times in
	$
	\mathcal{D}_k\ssp
  S_{\rho_N} \vert_{u_1 = \ldots = u_N = 1}
	$.
  Therefore, the leading order of asymptotics in $N$ of
  $\mathcal{D}_k\ssp S_{\rho_N} \vert_{u_1 = \ldots = u_N = 1} $
  is $N^{k+1}$, and it is obtained from the terms of the form \eqref{sum3} with $l_1 = k - m$ and $l_2 = \ldots = l_{k - m} = 0$.
	Note that the leading power $N^{k+1}$ precisely corresponds to taking the moments of the empirical measure $m[\rho_N]$, see
	\eqref{eq:random-measure-m-rho-definition}.

	As a conclusion, the leading asymptotics in $N$ of
	$
	\mathcal{D}_k\ssp
  S_{\rho_N} \vert_{u_1 = \ldots = u_N = 1}
	$
	is given by
	$$
	N^{k+1} \sum_{m = 0}^k \binom{k}{m} \frac{1}{(m + 1) !} \left( \sum_{i=1}^{m+1} \left. \frac{ u_i^k (\partial_1 F_{k} (u_i, u_1, \ldots, \hat u_i, \ldots, u_k))^{k - m} }{ \prod_{j=1 ; j \neq i}^{m+1} (u_i - u_j) } \right|_{u_1, u_2, \dots, u_{m+1}=1} \right).
	$$
	Applying \Cref{lem:main-sym}, we arrive at
	the fact that the expected moments of the empirical measure
	$m[\rho_N]$ converge to the moments $\tmmathbf{\mu}_k$ given by
	\eqref{eq:LLN-gen}.

	\medskip
	\noindent
	\textbf{Step 2.}
	It remains to show that the random moments
	of the empirical measure $m[\rho_N]$
	concentrate around their expected values.
	This would follow from
  \begin{equation}
    \lim_{N \rightarrow \infty} \operatorname{\mathbf{E}} \left[ \frac{1}{N^{k +
    1}} \sum_{i = 1}^N (\lambda_i + N - i)^k \right]^2 = (\tmmathbf{\mu}_k)^2
    \text{, \qquad for every } k \ge1, \label{16}
  \end{equation}
	and will guarantee the desired convergence of moments in probability.

	Up to a power of $N$,
	the left-hand side of \eqref{16} can be obtained by applying
	$\mathcal{D}_k$ to $S_{\rho_N}$ twice, and then setting $u_1
	= \dots = u_N = 1$.
  We have
  \begin{equation}
  \label{eq:Dk2}
  (\mathcal{D}_k)^2 S_{\rho_N} =\mathcal{D}_{2 k} S_{\rho_N} + \frac{1}{V_N
	 (\vec{u})}
	\sum_{\substack{m,n = 1 \\ m \neq n}}^N
	 (u_n \partial_n)^k (u_m
	 \partial_m)^k [V_N (\vec{u}) S_{\rho_N}].
	 \end{equation}
	The leading power of $N$ in
	$\mathcal{D}_{2 k} S_{\rho_N}$ is $N^{2k+1}$, and so it will not contribute to the
	limit in \eqref{16} which will be of the order $N^{2k+2}$.
  Next, in the second summand in \eqref{eq:Dk2}, only terms of the form
  \begin{equation}
    \frac{1}{V_N (\vec{u})}
		\sum_{\substack{m,n = 1 \\ m \neq n}}^N
		u_n^k u_m^k
    \partial_n^k \partial_m^k [V_N (\vec{u}) S_{\rho_N}] \Big\vert_{u_1=\dots= u_N=1}, \label{17}
  \end{equation}
  contribute to the degree $N^{2k+2}$. Expression \eqref{17} is a linear combination of terms
  \begin{equation}
    \frac{u_{a_0}^k u_{b_0}^k \partial_{a_0}^{k - \nu} \partial_{b_0}^{k -
    \mu} S_{\rho_N}}{(u_{a_0} - u_{a_1}) \ldots (u_{a_0} - u_{a_{\nu}})
    (u_{b_0} - u_{b_1}) \ldots (u_{b_0} - u_{b_{\mu}})}, \label{18}
  \end{equation}
  where $a_0 \neq b_0$, $a_i \neq a_j$, $b_i \neq b_j$ and $| \{ b_0 \} \cap
  \{ a_1, \ldots, a_{\nu} \} | + | \{ a_0 \} \cap \{ b_1, \ldots, b_{\mu} \} |
  \leq 1$. It is convenient to symmetrize these terms in $\{ u_{a_i} \}$ and $\{ u_{b_j} \}$.
	By \Cref{lem:lemma1},
	such symmetrizazions have a limit as $u_i \to 1$, $1 \le i \le N$.

	Analogously to, e.g., \cite[Theorem
	5.1]{GorinBufetov2013free}, the leading contribution in the
	large-$N$ expansion comes from the terms where the indices
	$a_0,\dots,a_\nu,b_0,\dots,b_\mu\in\{1,\dots,N\}$ are all
	distinct.
  This contribution can be written in the factorized form
\begin{multline*}
\left(
\frac{u_{a_0}^k \bigl( \partial_{a_0} \sqrt[N]{S_{\rho_N}} \bigr)^{k-\nu}}
     {(u_{a_0}-u_{a_1}) \ldots (u_{a_0}-u_{a_\nu})}
     + \ldots +
\frac{u_{a_\nu}^k \bigl( \partial_{a_\nu} \sqrt[N]{S_{\rho_N}} \bigr)^{k-\nu}}
     {(u_{a_\nu}-u_{a_0}) \ldots (u_{a_\nu}-u_{a_{\nu-1}})}
\right)
\\
\times
\left(
\frac{u_{b_0}^k \bigl( \partial_{b_0} \sqrt[N]{S_{\rho_N}} \bigr)^{k-\mu}}
     {(u_{b_0}-u_{b_1}) \ldots (u_{b_0}-u_{b_\mu})}
     + \ldots +
\frac{u_{b_\mu}^k \bigl( \partial_{b_\mu} \sqrt[N]{S_{\rho_N}} \bigr)^{k-\mu}}
     {(u_{b_\mu}-u_{b_0}) \ldots (u_{b_\mu}-u_{b_{\mu-1}})}
\right),
\end{multline*}
where each of the two factors is as in \eqref{sum3}.
  This implies \eqref{16}, which completes the proof of the theorem.
\end{proof}

A particular case of \Cref{th:LLN} recovers 
\cite[Theorem~5.1]{GorinBufetov2013free} (which we fomulated as
\Cref{theorem_moment_convergence} above):
\begin{corollary}
\label{cor:sanity}
In the notation of \Cref{th:LLN}, assume additionally that there exists a function $\mathsf{F}(u)$ such that
$$
F_k (u_1, \ldots, u_k) = \exp \left( \mathsf{F}(u_1) + \mathsf{F}(u_2) + \dots + \mathsf{F}(u_k) \right), \qquad \mbox{for any $k \in \mathbb{Z}_{\ge1}$.}
$$
Then the moments of the limiting measure can be written as
 \begin{equation}
  \tmmathbf{\mu}_k = \sum_{l = 0}^k \binom{k}{l} \frac{1}{(l + 1) !}
    \ssp \pa_u^{\ssp l}  \left[ (1+u)^k [\mathsf{F}'(1+u)]^{k-l} \right] \Big\vert_{u=0}.
		\label{eq:LLN-gen-cor}
  \end{equation}
\end{corollary}
\begin{proof}
We have
$$
\pa_1 F_k (u_1, \ldots, u_k) = \mathsf{F}'(u_1) F_k (u_1, \ldots, u_k),
$$
and
$$
\pa_1 F_{l+1} (1+u, 1+u w_{l+1},  \ldots, 1+u w_{l+1}^l) = \mathsf{F}'(1+u) \left( 1 + O \left( u^{l+1} \right) \right),
$$
due to the fact that the function $F_{l+1} (u_1, u_2, \dots, u_{l+1})$ is symmetric in all variables,
and thanks to \Cref{lem:roots-simplifying-sym}.
Therefore, the expression under $\pa_u^{\ssp l}$ in \eqref{eq:LLN-gen} can be simplified
for this special $F_k$, and
we arrive at \eqref{eq:LLN-gen-cor}.
\end{proof}


\begin{remark}
Compared to \cite{GorinBufetov2013free} and
\cite{bufetov2016fluctuations}, we work with the $N$-th root
of the Schur generating function, rather than with its
logarithm, as $N$ tends to infinity. It would be possible to
state our results in terms of the limit of the logarithm of
the Schur generating function; moreover, the
double-summation formula for the covariance in
\Cref{th:CLT-gen} below would then be slightly simpler.
However, our main application in this paper --- the Schur
measure with random parameters --- is more straightforward
to analyze with the current formulation.
\end{remark}

\begin{remark}
The use of complex numbers (roots of unity) in \eqref{eq:LLN-gen} is a key technical tool, even if it originally might look a bit artificial. We are able to write several equivalent formulas for this expression, which involve real coefficients only. In particular, we
have the following representation for the moments:
\begin{equation} 
		\label{ropes-2}
		\begin{split}
		&\tmmathbf{\mu}_k = \sum_{m = 0}^k \binom{k}{m} \frac{1}{(m + 1) !}
     \sum_{n = 1}^{m + 1} \binom{m + 1}{n}  
    \\&\hspace{18pt}\times 
		(- 1)^{n + 1}
		\sum_{l_1 + \cdots + l_n = m} \binom{m}{l_1, \ldots, l_n}
		\ssp
    \partial_n^{\ssp l_n} \ldots \partial_1^{\ssp l_1} \bigl[u_1^k (\partial_1 F (u_1,
    \ldots, u_k))^{k - m}\bigr] \Big\vert_{u_1 = \cdots = u_k = 1} . 
		\end{split}
\end{equation}
We omit the proof, because expression \eqref{eq:LLN-gen} for the 
moments $\tmmathbf{\mu}_k$ is
far more suitable for our purposes: simplifying
\eqref{ropes-2} for our main application --- the domino
tilings of the Aztec diamond --- would require substantial
additional work. We also believe that \eqref{eq:LLN-gen}
will be more convenient for other potential applications.
\end{remark}

\subsection{Free cumulants}

Theorem 5.1 in \cite{GorinBufetov2013free} is closely
related to free probability; see, e.g.,
\cite[Sections~1.4 and~1.5]{GorinBufetov2013free},
\cite{matsumoto2018moment}, and \cite{Collins2018}.
In this section, we derive an expression for the free cumulants
of the limiting measure whose moments are given by
formula \eqref{eq:LLN-gen}.
For basics on free cumulants, see \cite{Speicher2011}.


Let $\{ G_k (u_1, \ldots, u_k) \}_{k \ge1}$ be a sequence of
functions that are symmetric in all variables except possibly for the first one, and such that
\[ G_k (u_1, \ldots, u_k) = G_{k + 1} (u_1, \ldots, u_k, 1), \qquad k \ge 1. \]
They should be thought of as functions appearing in the right-hand side of \eqref{assumptioN1}, differentiated once with respect to $u_1$ (thus, no longer symmetric in $u_1$).

\begin{lemma}
  Let $\{ G_k (u_1, \ldots, u_k) \}_{k \ge1}$ be a sequence of
	functions as above. Then, for any $k,l \ge1$, and any $1\le s
	\le \min(k,l)$, we have
   \begin{equation*}
	 \partial_u^{\ssp s} G_{l + 1} (1 + u, 1 + w_{l + 1} u, \ldots, 1 +
     w_{l + 1}^l u) \big\vert_{u = 0}
		 =
		 \partial_u^{\ssp s} G_{k + 1} (1 + u, 1 +
     w_{k + 1} u, \ldots, 1 + w_{k + 1}^k u) \big\vert_{u = 0} .
	 \end{equation*}
\end{lemma}

\begin{proof}
Without loss of generality, assume that $k >l$. We need to prove that the Taylor coefficients up to degree $l$ of the functions
$$
G_{k + 1} (1 + u, 1 + w_{l + 1} u, \ldots, 1 + w_{l + 1}^l u,1,1, \dots,1)
$$
and
$$
G_{k + 1} (1 + u, 1 + w_{k + 1} u, \ldots, 1 + w_{k + 1}^k u)
$$
coincide. To see this, note that the Newton power sums of
the $k$-element sets
$(w_{k + 1}, \dots, w_{k + 1}^k)$
and
$(w_{l + 1}, \dots, w_{l + 1}^l, 0,\dots,0)$
coincide
up to degree $l$. Therefore, any symmetric polynomial in
these variables also has the same value on the two sets.
Owing to the symmetry with respect to all variables except
the first one, the Taylor coefficients involve only such
polynomials, which completes the proof of the lemma.
\end{proof}

Therefore, the moments
$\tmmathbf{\mu}_k$
\eqref{eq:LLN-gen} can be written as
\[ \tmmathbf{\mu}_k = \sum_{l = 0}^k \binom{k}{l} \frac{1}{(l + 1) !}
   \ssp
	 \partial_u^{\ssp l} [(1 + u)^k (\partial_1 F_{k + 1} (1 + u, 1 + w_{k +
   1} u, \ldots, 1 + w_{k + 1}^k u))^{k - l}] \Big\vert_{u = 0}, \]
the difference with \eqref{eq:LLN-gen} is that now only the function $F_{k+1}$ appears in the right-hand side.

Arguing as in \cite[Equation 6.2]{GorinBufetov2013free}, we can write this summation as
a contour integral:
$$
\frac{1}{2 \pi i (k + 1)} \oint_1 \frac{1}{u} \left( u\ssp \partial_1 F_{k +
   1} (u, 1 - w_{k + 1} + w_{k + 1} u, \ldots, 1 - w_{k + 1}^k + w_{k + 1}^k
   u) + \frac{u}{u - 1} \right)^{k + 1} d \text{} u,
$$
where we integrate over a small contour around $1$,
oriented counterclockwise. Under the change of variables $u = e^v$,
this integral becomes
\begin{equation}
  \frac{1}{2 \pi i (k + 1)} \oint_0 \left( e^v \partial_1 F_{k + 1} (e^v,
  1 - w_{k + 1} + w_{k + 1} e^v, \ldots, 1 - w_{k + 1}^k + w_{k + 1}^k e^v) +
  \frac{e^v}{e^v - 1} \right)^{k + 1} d \text{} v, \label{eq:m}
\end{equation}
where the contour $\mathcal{C}_1$ is now a small positively oriented circle
around $v = 0$.
This computation leads to the following:
\begin{proposition}
  The moments $(\tmmathbf{\mu}_k)_{k \ge1}$ of the limiting measure
  from \Cref{th:LLN} are given by
  \begin{multline}
		\label{eq:3}
		\tmmathbf{\mu}_k
		= \sum_{l=0}^k \binom{k}{l} \frac{1}{(l+1)!}
		\ssp\partial_v^{\ssp l}
		\biggl(
					\frac{e^v}{e^v-1}
			- \frac{1}{v}
\\+
			e^v \partial_1 F_{k+1}\big(
				e^v,\,
				1-w_{k+1}+w_{k+1}e^v,\,
				\ldots,\,
				1-w_{k+1}^k+w_{k+1}^k e^v
			\big)
		\biggr)^{k-l}
		\bigg\vert_{v=0},
	\end{multline}
  and its free cumulants $(\tmmathbf{c}_k)_{k \ge1}$ are given by
  \begin{multline*}
		\tmmathbf{c}_k
		=
		\frac{1}{(k-1)!}
		\partial_v^{k-1}
		\biggl(
					 \frac{e^v}{e^v-1}
			- \frac{1}{v}
\\+
			e^v \partial_1 F_{k+1}\big(
				e^v,\,
				1-w_{k+1}+w_{k+1}e^v,\,
				\ldots,\,
				1-w_{k+1}^k+w_{k+1}^k e^v
			\big)
		\biggr)
		\bigg\vert_{v=0}.
	\end{multline*}
\end{proposition}
\begin{proof}
Starting from the right-hand side of \eqref{eq:3} and
applying Cauchy's integral formula together with the
binomial theorem, we arrive at \eqref{eq:m}. This
establishes the claim for the moments. Having done that, we
can appeal to \cite[Lemma 4]{BufetovZografos2024} to obtain
the desired expression for the free cumulants.
\end{proof}

\subsection{Schur measures with random parameters}

In this section we address the specific form of Schur
generating functions that appears in domino tilings of Aztec
diamond with random edge weights.

\begin{proposition}
\label{LLN:tilings-fixed}
Let $\rho_N$, $N \ge1$, be a sequence of probability measures on $\GT_N$, such that their Schur generating function has a form
$$
S_{\mathbf{\rho}_N} (x_1,\dots,x_N) = \left( \operatorname{\mathbf{E}}_{\BB} \prod_{i=1}^N \left( 1 - \bb + x_i \bb \right) \right)^{M_N},
$$
where the distribution $\BB$ of the random variable $\bb$ is independent of $M$, and has
moments of all orders.
Assume that $\lim_{N \to \infty} M_N / N = a \in \R_{>0}$.
Then the random measure $m [\rho_N]$ converges, as $N \rightarrow \infty$, in
  probability, in the sense of moments to a deterministic probability measure
  $\tmmathbf{\mu}$ on $\mathbb{R}$, with moments $(\tmmathbf{\mu}_k)_{k\ge1}$ given by
\begin{equation}
\label{eq:LLN-contour-gen}
\tmmathbf{\mu}_k = \frac{1}{2\pi \sqrt{-1} (k+1) } \oint_{\mathcal{C}_1} \frac{d z}{z} \ssp \mathcal{F} (z)^{k+1},
\end{equation}
where
\begin{equation}
\label{eq:LLN-F-definition}
\mathcal{F} (z) \coloneqq \frac{z}{z-1} + a\ssp z\ssp \operatorname{\mathbf{E}}_{\BB} \frac{ \bb}{1-\bb +\bb z},
\end{equation}
and $\mathcal{C}_1$ is a contour oriented counterclockwise around $z=1$ that encloses no other poles of the integrand.
\end{proposition}
\begin{proof}
We have
\begin{equation*}
    \lim_{N \rightarrow \infty} \sqrt[N]{S_{\rho_N} (u_1,
    \ldots, u_k, 1^{N - k})} = \left( \operatorname{\mathbf{E}}_{\BB} \prod_{i=1}^k \left( 1 - \bb + u_i \bb \right) \right)^{a},
  \end{equation*}
	and this is our function $F_k(u_1,\ldots,u_k  )$.
Therefore, in the notation of \Cref{th:LLN}, we have
\begin{equation*}
    \pa_1 F_k (u_1, \dots, u_k) = a \operatorname{\mathbf{E}}_{\BB} \left(  \prod_{i=1}^k \left( 1 - \bb + u_i \bb \right) \right)^{a-1} \operatorname{\mathbf{E}}_{\BB} \left( \frac{\bb}{1 - \bb + u_1 \bb} \left( \prod_{i=1}^k \left( 1 - \bb + u_i \bb \right) \right) \right) .
  \end{equation*}
Noting that
$\prod_{i=1}^k \left( 1 - \bb + u_i \bb \right)$ is a symmetric function in all variables and applying \Cref{lem:roots-simplifying-sym}, we see that its nonzero degree terms in the $u_i$'s will not contribute after the symmetrization. This, only the expression
$$
a \left( \operatorname{\mathbf{E}}_{\BB} \left( \frac{\bb}{1 - \bb + u_1 \bb} \right) \right)
$$
contributes to the limit.
Applying \Cref{th:LLN} and the argument from \cite[Equation 6.2]{GorinBufetov2013free}, we arrive at the result.
\end{proof}

A key step in recovering the limiting density of the empirical measure from its moments is the analysis of the equation $\mathcal{F}(z)=y$. For the moment formula \eqref{eq:LLN-contour-gen}, it was shown, see e.g.\ \cite[Lemma~4.1 and Theorem~4.3]{BufetovKnizel2018}, that if, for each $y\in\mathbb{R}$, this equation has a unique complex root $z(y)$ lying in the upper half-plane, then the density of the limiting measure is given by $\frac{1}{\pi}\ssp\mathrm{Arg}\bigl(z(y)\bigr)$.

To apply this to the
Aztec diamond,
one needs a change of variables in the function $\mathcal{F}(z)=\mathcal{F}(z;a)$ \eqref{eq:LLN-F-definition}
and the corresponding equation. Namely, one needs to find the solution $z=z(\alpha,y)$
in the complex upper half-plane to the equation
\begin{equation}
	\label{eq:Aztec-diamond-equation-F-y}
	\mathcal{F}(z;\alpha^{-1}-1)=\alpha^{-1} y,
\end{equation}
where
$\alpha, y\in[0,1]$ serve as coordinates in the Aztec diamond.
A few examples of the limiting density
\begin{equation}
	\label{eq:limit-shape-density}
	(\alpha,y)\mapsto \frac{1}{\pi}\ssp \mathrm{Arg}\bigl(z(\alpha,y)\bigr)
\end{equation}
(which is a suitable linear transformation of the domino height function given in
\Cref{fig:height-function})
are presented in \Cref{fig:limit-shape-examples} (recall that the edge weights
are related to the parameters as $W_i=\bb_i/(1-\bb_i)$).
Note that for non-discrete distributions $\BB$, equation \eqref{eq:Aztec-diamond-equation-F-y}
can be solved only numerically.
On the other hand, the arctic curve has an exact parametrization,
expressing $(\alpha,y)$ as functions of the 
double root $z$ of the equation \eqref{eq:Aztec-diamond-equation-F-y}. 
We have used this parametrization to plot the arctic curves in \Cref{fig:limit-shape-examples}.

\begin{figure}[htpb]
\centering
\begin{tabular}{ccc}
\includegraphics[width=0.2\textwidth]{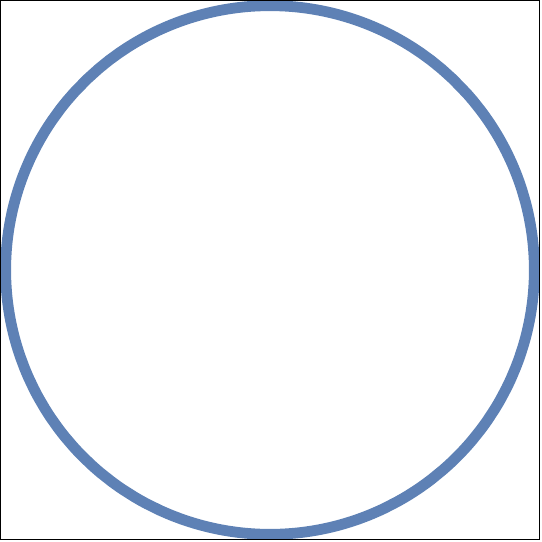} &
\includegraphics[width=0.2\textwidth]{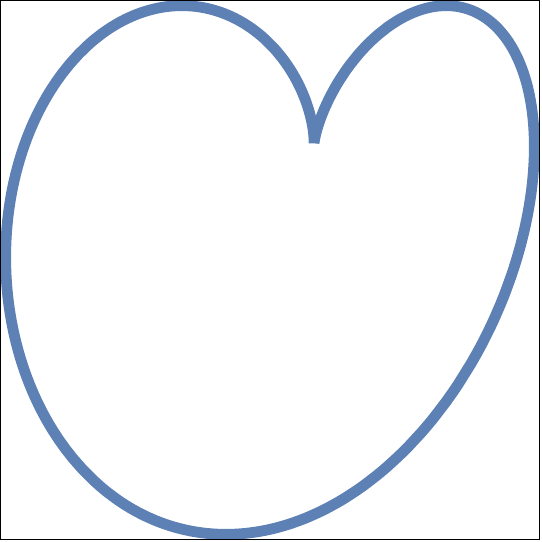} &
\includegraphics[width=0.2\textwidth]{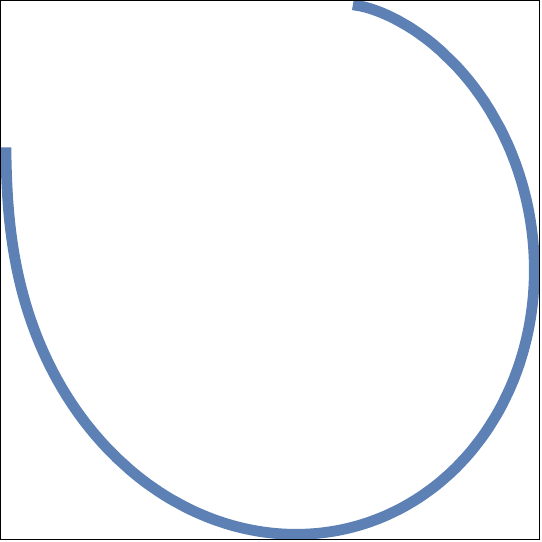} \\
\includegraphics[width=0.32\textwidth]{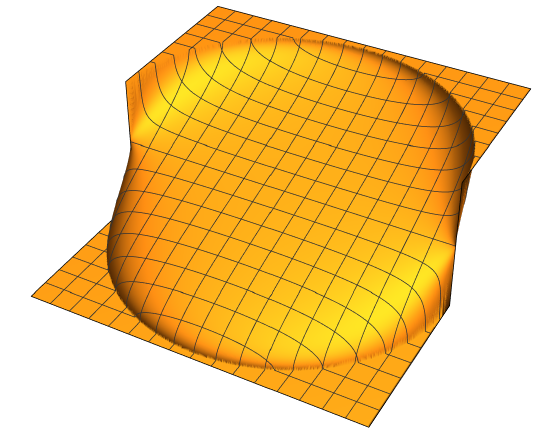} &
\includegraphics[width=0.32\textwidth]{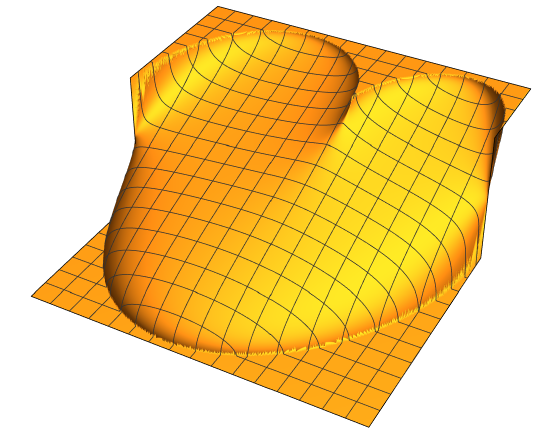} &
\includegraphics[width=0.32\textwidth]{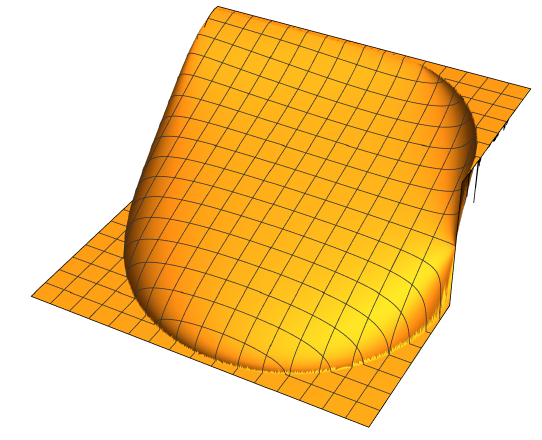}
\end{tabular}
\caption{Arctic curves (top row) and plots of the limiting density \eqref{eq:limit-shape-density} (bottom row) for
domino tilings of the Aztec diamond with constant edge weights $W_i\equiv 1$ (left),
i.i.d.\ Bernoulli edge weights with
$\operatorname{\mathbf{P}}(W_i=\frac12)=\operatorname{\mathbf{P}}(W_i=5)=\frac{1}{2}$ (middle),
and i.i.d.\ uniform edge weights on $[0,2]$ (right).
These plots are obtained from the equation \eqref{eq:Aztec-diamond-equation-F-y}:
arctic curves are the sets $(\alpha,y)$ for which \eqref{eq:Aztec-diamond-equation-F-y} has a double root in $z$, and
the limiting density comes from the argument of the complex root as in \eqref{eq:limit-shape-density}.
We see agreement with the samples in \Cref{fig:Bernoulli_and_fixed,fig:uniform_random_0_2_300}.}
\label{fig:limit-shape-examples}
\end{figure}







\subsection{Random Matrix degeneration}
\label{sec:rand-Matr-degen}

It is well known (see e.g.\ \cite{GuionnetMaida2005},
\cite{GorinBufetov2013free}) that the framework of Schur
generating functions admits a degeneration to results on
random matrices. We now present the degeneration of
\Cref{th:LLN} explicitly.

Let $A$ be a random Hermitian $N \times N$ matrix with
(possibly random) eigenvalues $\{\lambda_1 (A) \leq \ldots
\leq \lambda_N (A)\}$. Its \textit{Harish-Chandra transform}
(also known as a multivariate Bessel generating function) is
defined by
\begin{equation}
  \mathbb E [H \text{} C (x_1, \ldots, x_N ; \lambda_1 (A), \ldots, \lambda_N (A))] \coloneqq \mathbb E \int_{U (N)} \exp (\tmop{Tr} (A \text{} U \text{} B
  \text{} U^{\ast})) \ssp\tmmathbf{m}_N (d \text{} U), \label{intro-Haarmeasure}
\end{equation}
where $B$ is a deterministic diagonal matrix with eigenvalues $x_1, x_2, \dots, x_N\in \mathbb{C}$,
and the integration is with respect to a
Haar-distributed $N\times N$
unitary matrix
$U\in U(N)$.
We refer to
\cite{GorinBufetov2013free} for the definition and a detailed discussion,
see also \cite{BufetovZografos2024}
for recent applications of the transform.

\begin{theorem}
\label{th:LLN-rand-Mat}
Let $A=A_N$ be a random Hermitian matrix of size $N$. Assume that there exists a sequence
  of symmetric functions $\Phi_k (x_1, x_2, \dots, x_k)$, $k \ge1$, analytic
  in a complex neighborhood of $(0^k)$, such that for every fixed $k\ge1$ one has
  \begin{equation}
    \lim_{N \rightarrow \infty} \sqrt[N]{\operatorname{\mathbf{E}}[H \text{}
    C (x_1, \ldots, x_r, 0^{N - k} ; \lambda_1 (A), \ldots, \lambda_N (A))]}  =  \Phi_k (x_1, x_2, \dots, x_k), \label{eq:matrix-assumption}
  \end{equation}
uniformly in a complex neighborhood of $0^k$. Then the
  random atomic
	measure $\frac{1}{N} \sum_{i=1}^N \delta \left( \frac{\lambda_i (A)}{N} \right)$ converges, as $N \rightarrow \infty$, in
  probability, in the sense of moments to a deterministic probability measure
  $\tmmathbf{\nu}$ on $\mathbb{R}$, with moments $( \tmmathbf{\nu}_k)_{k \ge1}$ given by
  \begin{equation*}
  \tmmathbf{\nu}_k = \sum_{l = 0}^k \binom{k}{l} \frac{1}{(l + 1) !}
		 \ssp\pa_u^{\ssp l}  \left[ \left( \pa_1 \Phi_{l+1} (u, u
		 w_{l+1}, u w_{l+1}^2, \ldots, u w_{l+1}^l )
		 \right)^{k-l} \right] \bigg\vert_{u=0}.
  \end{equation*}
\end{theorem}
\begin{proof}
The simplest way to prove the result is to repeat all the arguments in the proof of \Cref{th:LLN}, adjusting them to accommodate the slightly different differential operators; see, e.g., \cite{GorinBufetov2013free} and \cite[Section~3]{BufetovZografos2024}. No substantial modifications are required, and the resulting formula differs only in that the factor $(1+u)^k$ is absent, owing to a minor distinction in the differential operators used.
\end{proof}

\section{Cental Limit Theorem}
\label{sec:CLT-gen}

\subsection{Statement}

In this section, we prove a general Central Limit Theorem
for probability measures on signatures defined by their
Schur generating functions. We continue to work in the
limit regime introduced in the previous section (which in an application
will correspond to domino tilings of the Aztec diamond with
i.i.d.\ edge weights coming from a fixed distribution).
This regime
leads to particle systems with fluctuations on
the scale~$1/\sqrt{N}$.

\begin{definition}
\label{def:SGF-appropriate}
The sequence of probability measures $\rho=\rho_N$ on $\GT_N$, $N\ge1$, is called \textit{asymptotically appropriate}, if
there exists a sequence of symmetric functions $F_k^{(\rho)}
(z_1, \ldots, z_k)$, $k\ge 1$, which are analytic in a complex
neighborhood of $1^k$, such that for every fixed $k$, one
has
  \begin{equation}
    \lim_{N \rightarrow \infty} \sqrt[N]{S_{\rho_N} (u_1,
    \ldots, u_k, 1^{N - k})} = F_k^{(\rho)} (u_1, \ldots, u_k), \label{assumptioN}
  \end{equation}
uniformly in a complex neighborhood of $1^k$.
\end{definition}

\begin{theorem}
\label{th:CLT-gen} Let $\rho = (\rho_N)_{N \ge1}$ be an asymptotically appropriate sequence of probability measures on $\GT_N$.
Then the vector of normalized moments \eqref{eq:p_k-rho-moments-in-text},
\begin{equation}
\label{eq:scale-CLT}
\left( \frac{ p_k^{(\rho_N)} - \raisebox{-1pt}{$\operatorname{\mathbf{E}}$} ( p_k^{(\rho_N)} ) }{N^{k+\frac12 } } \right)_{k \ge 1}
\end{equation}
converges to a mean zero Gaussian vector, with covariance given by
\begin{align}
&\lim_{N\to\infty}\frac{1}{N^{k+l+1}}
 \bigl(
   \operatorname{\mathbf{E}}\bigl(p_{k}^{(\rho_N)}
   p_{l}^{(\rho_N)}\bigr)-
   \operatorname{\mathbf{E}}\bigl(p_{k}^{(\rho_N)}\bigr)
   \operatorname{\mathbf{E}}\bigl(p_{l}^{(\rho_N)}\bigr)
 \bigr)                                                      \nonumber\\
&=\sum_{q=0}^{k-1}\sum_{r=0}^{l-1}
     \frac{k\,l}{(q+1)!\,(r+1)!}
     \binom{l-1}{r}\binom{k-1}{q}                            \nonumber\\
&\times\partial_1^{q}\partial_2^{r}\Bigl[
   \bigl(
     \partial_1\partial_2 F_{q+r+2}^{(\rho)}\bigl(
       1+x_1,1+x_2,
       1+x_1 w_{q+1},\dots,1+x_1 w_{q+1}^{q},
       1+x_2 w_{r+1},\dots,1+x_2 w_{r+1}^{r}\bigr)          \nonumber\\
&\quad-\partial_1 F_{q+1}^{(\rho)}\bigl(
       1+x_1,1+x_1 w_{q+1},\dots,1+x_1 w_{q+1}^{q}\bigr)
      \,\partial_2 F_{r+1}^{(\rho)}\bigl(
       1+x_2,1+x_2 w_{r+1},\dots,1+x_2 w_{r+1}^{r}\bigr)
   \bigr)                                                   \nonumber\\
&\qquad\qquad \times(1+x_{1})^{k}
   \bigl(\partial_1 F_{q+1}^{(\rho)}(
      1+x_1,1+x_1 w_{q+1},\dots,1+x_1 w_{q+1}^{q})
     \bigr)^{k-1-q}                                         \nonumber\\
&\qquad \qquad \times(1+x_{2})^{l}
   \bigl(\partial_2 F_{r+1}^{(\rho)}(
      1+x_2,1+x_2 w_{r+1},\dots,1+x_2 w_{r+1}^{r})
     \bigr)^{l-1-r}
 \Bigr]\biggr|_{x_1=x_2=0}.
 \label{eq:cov-general-statement}
\end{align}
\end{theorem}
\begin{remark}
In \Cref{th:CLT-gen} we permit the limiting Gaussian vector
to be identically zero. This situation arises, for instance,
when the theorem is applied to measures falling within the
limit regime of \cite{bufetov2016fluctuations} (and in particular, to
domino tilings of the Aztec diamond with deterministic weights or random weights with variance decaying as $1/M$).
The theorem
becomes nontrivial once one shows that the covariance at
the scale \eqref{eq:scale-CLT} remains non-zero in the
limit.
\end{remark}

\begin{remark}
\Cref{th:CLT-gen} can also be degenerated to a random matrix setting, similarly to \Cref{sec:rand-Matr-degen},
but we do not state the corresponding result explicitly.
\end{remark}

\subsection{First and second moment}

Throughout the rest of this section, we prove \Cref{th:CLT-gen}.
Its proof is close to that in
\cite{bufetov2016fluctuations}. In some aspects, it is even
easier since there are fewer terms that contribute to the
limiting covariance in our limit regime. Therefore, we do
not repeat all the details of \cite[Sections~5~and~6]{bufetov2016fluctuations}, but rather provide a sketch of
them and focus on those parts that differ from the argument
there. We also keep the same notation as in \cite{bufetov2016fluctuations}.

\begin{definition}
\label{def:degree-N}
For any $N \ge 1$, let $F_N (\vec{x})$,
$\vec{x}=(x_1,x_2, \dots, x_N)$,
be a function of~$N$~variables. For
$D\in \mathbb{Z}$, we will say that the sequence $(F_N)_{N\ge1}$
\textit{has
$N$-degree at most $D$}, if for any integer $s \ge 0$ (not depending
on $N$) and any indices $i_1, \dots, i_s$, we have
\begin{equation}
\label{eq:degree-def}
\lim_{N \to \infty} \frac{1}{N^D}  \pa_{i_1} \dots \pa_{i_s} F_N (\vec{x}) \big\vert_{\vec{x}=1} = c_{i_1, \dots, i_s},
\end{equation}
for some constants $c_{i_1, \dots, i_s}$. In particular, the limit
$$
\lim_{N \to \infty} \frac{1}{N^D} \left. F_N (\vec{x}) \right|_{\vec{x}=1}
$$
should exist (this corresponds to $s=0$).

Similarly, we will say that the sequence $(F_N)_{N\ge1}$
\textit{has $N$-degree less than $D$}, if for any $s \ge 0$ (not depending on $N$) and any indices $i_1, \dots, i_s$, we have
$$
\lim_{N \to \infty} \frac{1}{N^D}  \pa_{i_1} \dots \pa_{i_s} F_N (\vec{x}) \big\vert_{\vec{x}=1} = 0.
$$
\end{definition}

Let $\rho = \{ \rho_N \}$ be an asymptotically appropriate
sequence of measures on $\GT_N$ with the Schur generating
function $S_N(\vec x) = S_{\rho_N} (x_1, \dots, x_N)$.
For an integer $l>0$, let us introduce the function
\begin{equation}
\label{eq:def-Fl}
\mathcal{F}_{(l)} (\vec{x}) \coloneqq \frac{1}{S_N (\vec{x}) V_N (\vec{x}) } \sum_{i=1}^N \left( x_i \pa_i \right)^{l} V_N (\vec{x})
\ssp S_N (\vec{x}).
\end{equation}

\begin{lemma}
\label{lem:sum-Schur-Vand-afterDiff}
The functions $\mathcal{F}_{(l)} (\vec{x})$ have $N$-degree at most $l+1$.
\end{lemma}
\begin{proof}
The result follows from the argument in \Cref{sec:LLN}. We now give a brief, informal outline.

The power of $N$ can arise either from derivatives applied
to $S_N$ --- we use the identity $\partial_i S_N = N
(\sqrt[N]{S_N})^{\,N-1}\,\partial_i \sqrt[N]{S_N}$, and any
additional derivatives acting on $\partial_i \sqrt[N]{S_N}$
do not raise the power of $N$ --- or from the summations over
indices, since extra indices appear when we differentiate
the Vandermonde determinant. Because the operator involves
one summation and $l$ derivatives, the maximal power of $N$
is $l+1$. After we divide by $S_N(\vec{x})\,V_N(\vec{x})$,
further derivatives cannot increase this power, although,
unlike in \cite{bufetov2016fluctuations}, they do not
necessarily decrease it either.
\end{proof}

For positive integers $l_1,l_2$, let us define one more function via
\begin{equation}
\label{eq:def-GGfunc}
\begin{split}
\GG_{(l_1, l_2)} (\vec{x})
&\coloneqq
l_1
\sum_{r=0}^{l_1-1}
\binom{l_1-1}{r}
\sum_{\{a_1,\dots,a_{r+1}\}\subset[N]}
(r+1)!
\\
&\hspace{40pt}\times
\,Sym_{a_1,\dots,a_{r+1}}
\frac{
x_{a_1}^{l_1}\,
\pa_{a_1}\bigl[\mathcal{F}_{(l_2)}\bigr]\,
\bigl(\pa_{a_1}\sqrt[N]{S_N}\bigr)^{\,l_1-1-r}
}{
(x_{a_1}-x_{a_2})\dotsm(x_{a_1}-x_{a_{r+1}})
}.
\end{split}
\end{equation}
Here, $Sym_{a_1,\dots,a_{r+1}}$ denotes the symmetrization
over the indices $a_1, \dots, a_{r+1}$.
The following lemma clarifies the meaning of this function: it describes the covariance in our probability models.
\begin{lemma}
\label{lem:two-diff-terms}
For any positive integers $l_1, l_2$, we have
\begin{equation}
\label{eq:two-diff-terms}
\frac{1}{V_N S_N} \sum_{i_1=1}^N (x_{i_1} \pa_{i_1})^{l_1} \sum_{i_2=1}^N (x_{i_2} \pa_{i_2})^{l_2} \left[ V_N S_N \right] = \mathcal{F}_{(l_1)} (\vec{x}) \mathcal{F}_{(l_2)} (\vec{x}) + \GG_{(l_1, l_2)} (\vec{x}) + \tilde T (\vec{x}),
\end{equation}
where $\GG_{(l_1, l_2)} (\vec{x})$ has $N$-degree at most $l_1+l_2+1$, and $\tilde T (\vec{x})$ has $N$-degree at most $l_1+l_2$.
\end{lemma}

\begin{proof}
The proof is analogous to \cite[Lemma 5.7]{bufetov2016fluctuations}. Let us present it in a sketched form.
The left-hand side of \eqref{eq:two-diff-terms} can be written as
\begin{equation}
\label{eq:two-diff-terms-2}
\frac{1}{V_N S_N} \sum_{i_1=1}^N (x_{i_1} \pa_{i_1})^{l_1} \left[ V_N S_N \mathcal{F}_{(l_2)} (\vec{x}) \right].
\end{equation}
The left-hand side of \eqref{eq:two-diff-terms-2} has $l_1$ differentiations. If all of them are applied to $V_N$ or $S_N$, we obtain the first term in the right-hand side. If exactly one differentiation is applied to $\mathcal{F}_{(l_2)}$, this yields a contribution of degree $l_1+l_2+1$; the choice of which differentiation to apply gives rise to the factor $l_1$ in \eqref{eq:def-GGfunc}, while the other factors arise from choosing which differentiations are applied to the Vandermonde. If more than one differentiation is applied to $\mathcal{F}_{(l_2)}$, then the maximal possible degree is $l_1+l_2$, since these two (or more) differentiations do not increase the degree.
\end{proof}

\subsection{Several moments}

For \(s\in\mathbb{Z}_{\ge 1}\) and a subset
\(\{j_1,\dots ,j_p\}\subset\{1,2,\dots ,s\}\),
let \(\mathcal P^{s}_{j_1,\dots ,j_p}\) be the set of \emph{all
pairings} of
\(\{1,2,\dots ,s\}\setminus\{j_1,\dots ,j_p\}\).
The set \(\mathcal P^{s}_{j_1,\dots ,j_p}\)
is empty whenever the cardinality of
\(\{1,2,\dots ,s\}\setminus\{j_1,\dots ,j_p\}\) is odd.
Analogously, define \(\mathcal P^{2;s}_{j_1,\dots ,j_p}\)
as the set of \emph{all pairings} of
\(\{2,\dots ,s\}\setminus\{j_1,\dots ,j_p\}\).
For a pairing \(P\) we write \(\prod_{(a,b)\in P}\) for the
product over all pairs \((a,b)\) contained in \(P\).

\begin{proposition}
\label{lem:gauss-many-Schur}
For any $s,l_1,\ldots,l_s \in \mathbb{Z}_{\ge1}$, we have
\begin{equation*}
\begin{split}
&\frac{1}{V_N S_N} \sum_{i_1=1}^N (x_{i_1} \pa_{i_1})^{l_1}
\cdots
\sum_{i_s=1}^N (x_{i_s} \pa_{i_s})^{l_s} \left[ V_N S_N
\right]
\\
&
\hspace{10pt}=
\sum_{p=0}^s \sum_{ \{ j_1, \dots, j_p \} \in
[s] } \mathcal{F}_{(l_{j_1})} (\vec{x}) \dots
\mathcal{F}_{(l_{j_p})} (\vec{x}) \left( \sum_{ P \in
\mathcal P^s_{j_1, \dots, j_p} } \prod_{(a,b) \in P}
\GG_{(l_a, l_b)} (\vec{x}) + \tilde T_{j_1, \dots,
j_p}^{1;s} (\vec{x}) \right),
\end{split}
\end{equation*}
where $\tilde T_{j_1, \dots, j_p}^{1;s} (\vec{x})$ has
$N$-degree less than $\sum_{i=1}^s l_i - \sum_{i=1}^p l_{j_i} + \frac{s-p}{2}$.
\end{proposition}
\begin{proof}
The proof parallels that of \cite[Proposition 5.10]{bufetov2016fluctuations}. We now sketch the main ideas.
We argue by induction in $s$. For the induction step, one needs to analyze the expression
\begin{equation*}
\begin{split}
&
\frac{1}{V_N S_N} \biggl( \sum_{i_1=1}^N (x_{i_1} \pa_{i_1})^{l_1} \biggr) \biggl[ V_N S_N \sum_{p=0}^{s-1} \sum_{ \{ j_1, \dots, j_p \} \in [2;s] } \mathcal{F}_{(j_1)} (\vec{x}) \dots \mathcal{F}_{(j_p)} (\vec{x})
\\
&
\hspace{130pt}
\times \biggl( \sum_{ P \in \mathcal P^{2;s}_{j_1, \dots, j_p} } \prod_{(a,b) \in P} \GG_{(k_a, k_b)} (\vec{x}) + \tilde T_{j_1, \dots, j_p}^{2;s} (\vec{x}) \biggr) \biggr],
\end{split}
\end{equation*}
for any choice of the set of indices $J_{\mathrm{old}} \coloneqq \{ j_1, \dots, j_p \} \subset [2;s]$.
If all new differentiations act on \(V_N S_N\), they
generate a new function \(\mathcal{F}\). If exactly one
differentiation acts on an existing \(\mathcal{F}\), it
produces a new \(\GG\) in place of that \(\mathcal{F}\). Any
other allocation of differentiations does not yield a
sufficiently high \(N\)-degree, and thus does not contribute
to the leading term in the limit.
\end{proof}

For a positive integer $l$, define
\begin{equation}
\label{eq:def-e-l-const}
E_l \coloneqq \mathcal{F}_{(l)} (1^N)
=
\frac{1}{V_N S_N} \sum_{i=1}^N (x_{i} \pa_{i})^{l} \ssp
V_N S_N \biggl\vert_{\vec{x}=1}.
\end{equation}
This is the expectation of the $l$-th moment of the probability measure with the Schur generating function $S_N$.

\begin{lemma}
\label{lem:covar-general}
For any $s,l_1,\ldots,l_s\in \mathbb{Z}_{\ge1} $,
we have
\begin{multline*}
\frac{1}{V_N S_N}
\left( \sum_{i_1=1}^N (x_{i_1} \pa_{i_1})^{l_1} - E_{l_1} \right)
\left( \sum_{i_2=1}^N (x_{i_2} \pa_{i_2})^{l_2} - E_{l_2} \right)
\times
\dots
\\
\times
\left( \sum_{i_s=1}^N (x_{i_s} \pa_{i_s})^{l_s} - E_{l_s} \right) V_N S_N \bigg\vert_{\vec{x}=1} = \sum_{ P \in \mathcal P^{s}_{\emptyset} } \prod_{(a,b) \in P} \GG_{(l_a, l_b)} (\vec{x}) \bigg\vert_{\vec{x}=1} + \tilde T_{\emptyset} (\vec{x}) \bigg\vert_{\vec{x}=1},
\end{multline*}
where $\tilde T_{\emptyset} (\vec{x})$ has $N$-degree less than $\sum_{i=1}^s l_i+\frac{s}{2}$.
\end{lemma}
\begin{proof}
The derivation of this lemma from \Cref{lem:gauss-many-Schur} follows the standard combinatorial argument that relates moments and cumulants of random variables; it is identical to the proof of \cite[Lemma~5.11]{bufetov2016fluctuations}.
\end{proof}

\subsection{Computation of covariance}
Here we finalize the proof of \Cref{th:CLT-gen} by computing the
limiting covariance. We need to establish formula
\eqref{eq:cov-general-statement}.


We the indices $k,l\in \mathbb{Z}_{\ge1}$ in \eqref{eq:cov-general-statement}.
By \Cref{lem:two-diff-terms}, the left-hand side of \eqref{eq:cov-general-statement} before the limit
and without the factor $N^{-k-l-1}$ is equivalent
(in the sense of the top degree in $N$) to the expression
\begin{multline*}
\GG_{(k, l)} (1^N) = k \sum_{r=0}^{k-1}
\binom{k-1}{r} \sum_{\{a_1, \dots, a_{r+1} \} \subset [N]
} (r+1)! \\ \times Sym_{a_1, \dots, a_{r+1}} \frac{
x_{a_1}^{k} \pa_{a_1} \left[ \mathcal{F}_{(l)} \right]
\left( N \pa_{a_1} \sqrt[N]{S_N} \right)^{k-1-r}}
{(x_{a_1} -x_{a_2} ) \dots (x_{a_1}-x_{a_{r+1}})}
\bigg\vert_{x_1=\ldots=x_N=1}.
\end{multline*}
Recall that $\mathcal{F}_{(l)} (\vec{x})$ can also be written as a sum of symmetrizations over indices. Using this, we obtain
\begin{align*}
&\GG_{(k,l)} (1^N)
\approx k \sum_{q=0}^{k-1}
\sum_{\{a_1, \dots, a_{q+1} \}  \subset [N] } \binom{k-1}{q} \ssp (q+1)! \ssp\ssp {Sym}_{a_1, \dots, a_{q+1}} \frac{ x_{a_1}^{k} \left( N \partial_{a_1} \sqrt[N]{S_N} \right)^{k-1-q}} {(x_{a_1} -x_{a_2} ) \cdots (x_{a_1}-x_{a_{q+1}})}
\\
&\quad
\times\ssp \partial_{a_1} \bigg[ \sum_{r=0}^{l}
\sum_{\{b_1, \dots, b_{r+1} \} \subset [N] } \binom{l}{r} (r+1)! \ssp\ssp {Sym}_{b_1, \dots, b_{r+1}} \frac{x_{b_1}^l \left( N \frac{\partial_{b_1} \sqrt[N]{S_N}}{\sqrt[N]{S_N}} \right)^{l-r}}{(x_{b_1}- x_{b_2}) \cdots (x_{b_1}- x_{b_{r+1}})} \bigg]
\Bigg\vert_{x_1=\ldots=x_N=1}.
\end{align*}
Observe that the terms of maximal $N$-degree in this
expression (namely, of degree $N^{k+l+1}$),
arise precisely when
$\{a_1, \dots, a_{q+1}\} \cap \{b_1, \dots, b_{r+1}\} =
\varnothing$. In this case the outer differentiation
$\partial_{a_1}$ acts only on the factor
$(\partial_{b_1}\sqrt[N]{S_N})/\sqrt[N]{S_N}$ as
\begin{multline*}
\pa_{a_1} \left( \frac{\pa_{b_1} \sqrt[N]{S_N}}{\sqrt[N]{S_N}} \right)^{l-r} \\
= (l-r) \left( \frac{\pa_{b_1} \sqrt[N]{S_N}}{\sqrt[N]{S_N}} \right)^{l-r-1} \frac{(\pa_{a_1} \pa_{b_1} \sqrt[N]{S_N})\cdot \sqrt[N]{S_N} - (\pa_{b_1} \sqrt[N]{S_N})\cdot (\pa_{a_1} \sqrt[N]{S_N}) }{\sqrt[N]{S_N}^2}.
\end{multline*}
Applying \Cref{lem:main-sym} twice, passing to the limit, and performing minor transformations (which coincide with those performed in \cite[Proposition~6.4, case~1]{bufetov2016fluctuations}), we arrive at the right-hand side of the desired expression
\eqref{eq:cov-general-statement}.
This completes the proof of \Cref{th:CLT-tilings}.

\begin{remark}
	While the proof of Gaussianity is almost identical to that
	in \cite{bufetov2016fluctuations}, the computation of the
	covariance was more involved in the setting of
	that previous work. In our scaling regime,
	some of the terms considered in
	\cite{bufetov2016fluctuations} dominate the others, so we
	needed to evaluate only the contribution of these dominant
	terms, which reduced the workload.
\end{remark}

\section{CLT for domino tilings in iid environment}

We now apply the results of the previous
\Cref{sec:LLN,sec:CLT-gen} to the case of
domino tilings of Aztec diamond
with i.i.d.\ one-periodic weights
as in \Cref{fig:AztecWeights-intro}.

\begin{theorem}
  \label{th:CLT-tilings}
	 Let $\rho = (\rho_N)_{N \ge1}$ be measures on signatures
	as in \Cref{LLN:tilings-fixed}. Then the vector
\begin{equation}
\label{eq:scale-CLT-2}
\biggl( \frac{ p_k^{(\rho_N)} -
\raisebox{-1pt}{$\operatorname{\mathbf{E}}$}
\bigl[ p_k^{(\rho_N)} \bigr] }{N^{k+\frac12 } } \biggr)_{k
\ge 1}
\end{equation}
converges, as $N\to\infty$, to a Gaussian vector with zero
mean, and covariance given by
\begin{equation}
\label{eq:cov-iid-answer}
\begin{split}
	\lim_{N \to \infty} \frac{\operatorname{Cov}
	\bigl( p_{k}^{(\rho_N)}, p_{l}^{(\rho_N)}\bigr)}
	{N^{k+l+1}}
	&
	=
	\frac{1}{(2 \pi \ii)^2} \oint_{|z|=\varepsilon} \oint_{|w|=2 \varepsilon} \left( \frac{1}{z} +1 + (1+z) \ssp \mathsf{F} (1+z) \right)^{k}
	\\&
	\hspace{20pt}\times \left( \frac{1}{w} +1 + (1+w) \ssp \mathsf{F} (1+w) \right)^{l}\mathsf{G}(1+z,1+w)\, dz\ssp dw,
\end{split}
\end{equation}
where the integration contours are counterclockwise,
$\varepsilon \ll 1$, and the functions $\mathsf{F}(z)$
and $\mathsf{G}(z,w)$ are given by
\begin{equation}
\label{eq:F-G-def-s7}
\mathsf{F} (z) \coloneqq \operatorname{\mathbf{E}}_{\BB} \left( \frac{\bb}{1-\bb + \bb z} \right),
\qquad
\mathsf{G}(z,w)\coloneqq \operatorname{Cov}_{\BB} \left( \frac{\bb}{1-\bb + \bb z}, \frac{\bb}{1-\bb + \bb w} \right).
\end{equation}
\end{theorem}
Recall that the subscript $\BB$ in the expectation
and covariance in \eqref{eq:F-G-def-s7} indicates that these
averages are not taken with respect to a random signature,
but rather with respect to the distribution $\BB$
defining the random environment.


\begin{proof}[Proof of \Cref{th:CLT-tilings}]
Note that the sequence $\rho_N$ as in the
hypothesis of the theorem
is appropriate in the sense of \Cref{def:SGF-appropriate}.
We have
\begin{equation*}
    \lim_{N \rightarrow \infty} \sqrt[N]{S_{\rho_N} (u_1,
    \ldots, u_k, 1^{N - k})} = \left( \operatorname{\mathbf{E}}_{\BB} \prod_{i=1}^k \left( 1 - \bb + u_i \bb \right) \right)^{a} = F_k^{(\rho)} (u_1, \dots, u_k), \label{eq:n-th-root-asym-2}
  \end{equation*}
and
$$
\pa_1 F_k^{(\rho)} (u_1, \dots, u_k) = a \left( \operatorname{\mathbf{E}}_{\BB} \prod_{i=1}^k \left( 1 - \bb + u_i \bb \right) \right)^{a-1} \operatorname{\mathbf{E}}_{\BB} \left( \frac{\bb}{1-\bb+\bb u_1} \prod_{i=1}^k \left( 1 - \bb + u_i \bb \right) \right).
$$
Differentiating further, we obtain
\begin{equation*}
\begin{split}
&\pa_1 \pa_2 F_k^{(\rho)} (u_1, \dots, u_k)
= a (a-1) \left( \operatorname{\mathbf{E}}_{\BB} \prod_{i=1}^k \left( 1 - \bb + u_i \bb \right) \right)^{a-2}
\\ &
\hspace{40pt} \times \operatorname{\mathbf{E}}_{\BB} \left( \frac{\bb}{1-\bb+\bb u_2} \prod_{i=1}^k \left( 1 - \bb + u_i \bb \right) \right)
\operatorname{\mathbf{E}}_{\BB} \left( \frac{\bb}{1-\bb+\bb u_1} \prod_{i=1}^k \left( 1 - \bb + u_i \bb \right) \right)
\\ &
\hspace{20pt} + a \left( \operatorname{\mathbf{E}}_{\BB} \prod_{i=1}^k \left( 1 - \bb + u_i \bb \right) \right)^{a-1}
\operatorname{\mathbf{E}}_{\BB} \left[ \frac{\bb^2}{(1-\bb +u_1 \bb) (1 - \bb +u_2 \bb) } \prod_{i=1}^k \left( 1 - \bb + u_i \bb \right) \right].
\end{split}
\end{equation*}
We apply \Cref{th:CLT-gen}. It establishes the convergence of moments to a Gaussian vector, and it remains to show that the general formula \eqref{eq:cov-general-statement} for covariance vastly simplifies in our case of Schur measures with random parameters.

Indeed, by \Cref{lem:roots-simplifying-sym}, the function
$\prod_{i=1}^k (1-\bb+\bb u_i)$, which is symmetric in all
variables, does not contribute to the limit beyond its free
term (which is equal to $1$).
One can check
that the resulting covariance is given by
\begin{multline*}
\sum_{q=0}^{k-1} \sum_{r=0}^{l-1} \frac{k l}{ (q+1)!\ssp (r+1)!}
\ssp
\binom{l-1}{r} \binom{k-1}{q}\ssp
 \pa_1^{\ssp q} \pa_2^{\ssp r} \Bigl[ \mathsf{G} (1+x_1,1+x_2)
\\
\times
(1+x_{1})^k
\bigl( \mathsf{F} (1+x_1) \bigr)^{k-1-q}
(1+x_{2})^l
\bigl( \mathsf{F} (1+x_2) \bigr)^{l-1-r}
\Bigr]
\bigg\vert_{x_1=0, x_2=0}.
\end{multline*}
Using the Cauchy integral formula and the binomial theorem, this expression turns into the desired
right-hand side of \eqref{eq:cov-iid-answer}.
\end{proof}

\begin{remark}
Similarly to \Cref{sub:covariance-structure-decreasing-variance}, one can straightforwardly establish a multilevel analogue of \Cref{th:CLT-tilings} for domino tilings of the Aztec diamond
with random edge weights,
in the spirit of \Cref{th:main-decreasing-multilevel} and \cite[Theorem~2.11]{bufetov2016fluctuations}.
This would produce a
Brownian motion interpretation of
the multilevel version of the
covariance in \eqref{eq:cov-iid-answer}.
\end{remark}

\input{SORTED_June2025_random_edge_Aztec.bbl}

\end{document}

%% file: SORTED_June2025_random_edge_Aztec.bbl
\newcommand{\etalchar}[1]{$^{#1}$}